\documentclass[11pt, reqno]{amsart}
\usepackage{etex}
\usepackage{amsmath, amssymb, latexsym, enumerate,  graphicx, MnSymbol}
\usepackage{amscd,mathrsfs,epic,empheq,float}
\usepackage{bbm}
\usepackage[all]{xy}
 \usepackage{pdfsync}
 \usepackage{todonotes}
\usepackage{tikz}
\usetikzlibrary{cd}
\usepackage{color}
\usetikzlibrary{matrix,arrows,decorations.markings,decorations.pathmorphing}
\usetikzlibrary{fadings}
\usepackage{graphicx}
\usepackage[vcentermath]{youngtab}
\usepackage{multirow}
\usepackage{stmaryrd}
\usepackage{youngtab}
\usepackage{young}
\usepackage{xcolor}
\usepackage{placeins}
\usepackage{hyperref}
\theoremstyle{definition}
\newtheorem{thm}{Theorem}
\newtheorem{prop}[thm]{Proposition}

\newtheorem{lem}[thm]{Lemma}

\newtheorem{defn}[thm]{Definition}

\newtheorem{fact}[thm]{Fact}
\newtheorem{rem}[thm]{Remark}
\newtheorem{ex}[thm]{Example}

\newtheorem*{first-step}{First Step}
\newtheorem*{third-step}{Third Step}
\newtheorem*{second-step}{Second Step}

\newtheorem*{main idea}{Main idea}

\newtheorem*{theorem*}{Theorem}

\newcommand{\op}{\operatorname}
\newcommand{\tightoverset}[2]{%
  \mathop{#2}\limits^{\vbox to -.5ex{\kern-0.75ex\hbox{$#1$}\vss}}}
  
\newcommand{\tightunderset}[2]{%
  \mathop{#2}\limits_{\vbox to -.5ex{\kern-0.75ex\hbox{$#1$}\vss}}}

\newcommand\smvee{\raise0.9ex\hbox{$\scriptscriptstyle\vee$}}

\newcommand{\T}{{\rm T}}
\newcommand{\E}{{\rm E}}
\newcommand{\R}{{\rm R}}

\newcommand{\s}{{\rm S}}

\newcommand{\N}{{\rm N}}

\newcommand{\p}{{\rm P}}
\newcommand{\A}{{\rm A}}

\newcommand{\li}{{\rm L}}

\newcommand{\Q}{{\rm Q}}

\newcommand{\slth}{\mathfrak{sl}(3, \mathbb{C})}

\date{}                                           
\newcount\cols
{\catcode`,=\active\catcode`|=\active
 \gdef\Young#1{\hbox{$\vcenter
 {\mathcode`,="8000\mathcode`|="8000
  \def,{\global\advance\cols by 1 &}%
  \def|{\cr
        \multispan{\the\cols}\hrulefill\cr
        &\global\cols=2 }%
  \offinterlineskip\everycr{}\tabskip=0pt
  \dimen0=\ht\strutbox \advance\dimen0 by \dp\strutbox
  \halign
   {\vrule height \ht\strutbox depth \dp\strutbox##
    &&\hbox to \dimen0{\hss$##$\hss}\vrule\cr
    \noalign{\hrule}&\global\cols=2 #1\crcr
    \multispan{\the\cols}\hrulefill\cr%
   }
 }$}}
\gdef\Skew(#1:#2){\hbox{$\vcenter
{\mathcode`,="8000\mathcode`|="8000
  \dimen0=\ht\strutbox \advance\dimen0 by \dp\strutbox
  \def\boxbeg{\vbox
    \bgroup\hrule\kern-0.4pt\hbox to\dimen0\bgroup\strut\vrule\hss$}%
  \def\boxend{$\hss\egroup\hrule\egroup}%
  \def,{\boxend\boxbeg}%
  \def|##1:{\boxend\vrule\egroup\nointerlineskip\kern-0.4pt
    \moveright##1\dimen0\hbox\bgroup\boxbeg}%
  \def\\##1\\##2:{\boxend\vrule\egroup\nointerlineskip\kern-0.4pt
    \kern ##1\dimen0\moveright##2\dimen0\hbox\bgroup\boxbeg}%
  \moveright#1\dimen0\hbox\bgroup\boxbeg#2\boxend\vrule\egroup
 }$}}
}

\title{A non-levi branching rule in terms of Littelmann paths}
\author{Bea Schumann }
\address{Mathematical Institute, University of Cologne, Weyertal 86-90, 50931 Cologne, Germany}
\email{bschumann@math.uni-koeln.de}

\author{Jacinta Torres}
\address{Karlsruhe Institute for Technology, Department of Mathematics, Institute for Algebra and Geometry, Englerstr. 2 Mathebau (20.30), 76131 Karlsruhe, Germany}
\email{jacinta.torres@kit.edu}

\subjclass[2010]{20G05 (primary), 05E05, 05E10 (secondary)}

\begin{document}
\maketitle

\begin{abstract}
We prove a conjecture of Naito-Sagaki about a branching rule for the restriction of irreducible representations of $\mathfrak{sl}(2n,\mathbb{C})$ to $\mathfrak{sp}(2n,\mathbb{C})$. The conjecture is in terms of certain Littelmann paths, with the embedding given by the folding of the type $A_{2n-1}$ Dynkin diagram. So far, the only known non-Levi branching rules in terms of Littelmann paths are the diagonal embeddings of Lie algebras in their product yielding the tensor product multiplicities.
\end{abstract}

\section*{Introduction}

Given a complex simple Lie algebra $\mathfrak{g}$, a finite-dimensional representation $V$ of $\mathfrak{g}$ and a complex reductive subalgebra $\widetilde{\mathfrak{g}}\subset \mathfrak{g}$ we have a natural action of $\widetilde{\mathfrak{g}}$ on $V$ by restricting the action of $\mathfrak{g}$. Under this restriction the property of irreducibility is not preserved in most cases. It is a classical problem in representation theory to determine the multiplicities of irreducible representations of $\widetilde{\mathfrak{g}}$ as direct summands of an irreducible $\mathfrak{g}$-representation $V$ under restriction, called the branching problem. A formula determining these restriction multiplicities is called a branching rule. In this work we prove a new branching rule for the restriction of $\mathfrak{sl}(2n,\mathbb{C})$ to $\mathfrak{sp}(2n,\mathbb{C})$ in terms of Littelmann paths which was conjectured by Naito and Sagaki in \cite{branchconj}.\\

Let $\mathfrak{h} \subset \mathfrak{b} \subset \mathfrak{g}$ be fixed Cartan and Borel subalgebras. Let $\mathfrak{h}^{*}_{\mathbb{R}}$ be the real span of the integral weight lattice. Consider the set $\Pi$ of piecewise linear paths $\pi: [0,1] \rightarrow \mathfrak{h}^{*}_{\mathbb{R}}$ starting at the origin and ending at an integral weight. To each simple root $\alpha$, Littelmann \cite{pathsandrootoperators} assigned root operators $f_{\alpha}$ and $e_{\alpha}$  partially defined on the set $\Pi$. Fix a path $\pi^{+} \in \Pi$ completely contained in the dominant Weyl chamber. Such a path is called \textit{dominant}. The subset $\mathcal{P}(\pi^{+}) \subset \Pi$ obtained by successive application of the root operators to the path $\pi^{+}$ is a model for the simple representation $\li(\lambda)$ of highest weight $\lambda = \pi^{+}(1)$: the sum over the endpoints of all paths in $\mathcal{P}(\lambda)$ is the character of $\li(\lambda)$, the Littlewood-Richardson rule is generalised in a natural way, and it is possible to describe the restriction of representations to Levi subalgebras, simply by considering a subset of the hyperplanes that define the dominant Weyl chamber (\cite{littelmannlrrule}).\\

Consider a reductive subalgebra $\widetilde{\mathfrak{g}}$ with a choice of Cartan subalgebra $\widetilde{\mathfrak{h}}$ such that
$$\widetilde{\mathfrak{h}}\subset \mathfrak{h}.$$
A path $\pi:[0,1] \rightarrow \mathfrak{h}^{*}_{\mathbb{R}}$ may be restricted to the path
$\op{res}(\pi): [0,1] \rightarrow (\widetilde{\mathfrak{h}}_{\mathbb{R}})^{*}$ via $\op{res}(\pi)(t):= \pi(t)|_{\widetilde{\mathfrak{h}}_{\mathbb{R}}}$.

We say that the path $\pi^{+}$ \textit{is adapted to the branching} if

\begin{align*}
\op{res}^{\mathfrak{g}}_{\widetilde{\mathfrak{g}}} (\li(\lambda)) = \underset{\delta \in \op{domres}(\lambda)}{\bigoplus} \widetilde{\li}(\delta(1)),
\end{align*}

\noindent
where $\op{domres}(\lambda)$ is the set of paths $\op{res}(\pi)$ that are dominant, for some choice of simple roots of $\tilde{\mathfrak{g}}$, and for $\pi \in \mathcal{P}(\pi^{+})$. For instance, any dominant path is adapted to any Levi. A path obtained by concatenation of two dominant paths is adapted to $\mathfrak{g} \subset \mathfrak{g} \times \mathfrak{g}$.\\

In the case of $\mathfrak{g} = \mathfrak{sl}(2n,\mathbb{C})$, the highest weight $\lambda$ may be naturally interpreted as a partition. Let $\op{SSYT}(\lambda)$ be the set of semi-standard Young tableaux of shape $\lambda$ in the ordered alphabet $\mathcal{A}_{2n} = \left\{1 < \ldots < 2n\right\}$. This set can be interpreted as a Littelmann path model $\mathcal{P}(\pi^{+}_{SSYT})$ for $\li(\lambda)$, where $\pi^{+}_{SSYT}$ is a path associated to the semi-standard Young tableau of shape $\lambda$ with entries only $i$'s in row $i$. Our main result reads as follows.

\begin{thm}\label{mainbranchingthm}
\label{mainbranchingthmintro}
The path $\pi^{+}_{SSYT}$ is adapted to $\mathfrak{sp}(2n,\mathbb{C}) \subset \mathfrak{sl}(2n,\mathbb{C})$, with the embedding given by the folding automorphism of the Dynkin diagram of type $\A_{2n-1}$. 
\end{thm}

\noindent 
The obtained branching rule can be expressed in terms of tableaux as follows. Given a semistandard Young tableau of shape $\lambda$ of at most $2n$ parts, replace each letter $i > n$ by $\overline{2n - i + 1}$, to produce a new tableau, $\op{res}(\T)$. Read the word of this new tableau from right to left and top to bottom. At each step $j$, let $\mu^{j}_{i}$ equal the number of $i's$ minus the number of $\overline{i}'s$ in the word up to that point. Then $\op{res}(\T)$ belongs to $\op{domres}(\lambda)$ if and only if $\mu_{j} = (\mu^{1}_{j}, \ldots, \mu^{n}_{j})$ is a partition for all steps $j$. \\

In order to prove our main Theorem \ref{mainbranchingthm} we use a branching rule obtained by Sundaram in \cite{sundaram} expressing the branching multiplicities of the restriction from $\mathfrak{sl}(2n,\mathbb{C})$ to $\mathfrak{sp}(2n,\mathbb{C})$ in terms of a subclass of Littlewood-Richardson tableaux, called here Littlewood-Richardson Sundaram tableaux (see Section \ref{skewsection} for the precise definition). The multiplicity of the $\mathfrak{sp}(2n,\mathbb{C})$ irreducible module $\tilde \li(\mu)$ in $\op{res}(\li(\lambda))$ is given by the cardinality $|\op{LRS}(\lambda, \mu)|$, where $\op{LRS(\lambda, \mu)}$ is the set of Littlewood-Richardson Sundaram tableaux of skew shape $\lambda/\mu$. So what we do in our proof is to establish a bijection

\begin{align*}
\op{domres}(\lambda,\mu) \overset{\cong}{\longrightarrow} \op{LRS}(\lambda,\mu).
\end{align*}

For this we use a bijection due to Berele and Sundaram, in spirit analogous to the Robinson-Schensted-Knuth correspondence, between so-called up-down sequences and pairs $(\Q, \li)$, where $\Q$ is a standard tableau of shape $\lambda$ and $\li$ is a Littlewood-Richardson Sundaram tableau.

We associate to an element $\T\in \op{domres}(\lambda,\mu)$ an up-down sequence $\mu_{\T}$ which is then sent to $\li_{\T}$ in the pair $(Q_{\T}, \li_{\T})$ obtained via the bijection of Berele and Sundaram. We show that $\Q_{\T}$ has shape $\lambda$ and depends only on this shape, this implies the injectivity, and a case by case analysis shows that it is also surjective.

Theorem \ref{mainbranchingthm} leads to the natural question of whether the conjecture is true for other path models for $\li(\lambda)$. In contrast to the Levi case there exist path models for which Theorem \ref{mainbranchingthm} is false.

Let $\lambda$ be a stable weight (see Definition \ref{stable}). In this case the set of Littlewood-Richardson Sundaram tableaux of skew shape $\lambda/ \mu$ coincides with the classical set of Littlewood-Richardson tableaux of this skew shape. In Sections \ref{pol1} and \ref{domresviainequalities} we give a bijection between both, the set of paths in $\op{domres}(\lambda)$ with endpoint $\mu$, and the set of Littlewood-Richardson-Sundaram tableau of skew shape $\lambda/\mu$, with lattice points of a convex polytope.

We conclude the paper with several open problems.

\section{Words and Paths}
\label{wordsandpaths}

For a positive integer $m \in \mathbb{Z}_{\geq 1}$, let $\mathfrak{h} \subset \mathfrak{b} \subset \mathfrak{sl}(m, \mathbb{C})$ be the Cartan subalgebra of diagonal matrices, respectively the Borel subalgebra of upper triangular matrices in the special linear Lie algebra of traceless, complex $m \times m$ matrices. Let $\lambda \in \mathfrak{h}^{*}$ be an integral weight that is dominant with respect to this choice. Let $\varepsilon_{i} \in \mathfrak{h}^{*}$ be defined by $\varepsilon_{i}(\op{diag}(a_{1}, \ldots, a_{m-1})) = a_{i}$. We write $\omega_{1}, \ldots, \omega_{m - 1}$, $\omega_{i} = \varepsilon_{1} + \ldots + \varepsilon_{i}$, for the fundamental weights in $\mathfrak{h}^{*}$ and $\mathfrak{h}^{*}_{\mathbb{R}}$ for the real span of the fundamental weights.

To a dominant integral weight $\lambda = a_{1}\omega_{1} + \ldots + a_{m-1}\omega_{m-1}$ is associated a Young diagram of shape the partition $ (a_{1}+ \ldots + a_{m-1}, \ldots , a_{m-1})$ with $a_{k}$ columns of length $k$. We use the same symbol $\lambda$ to denote the dominant weight and the partition. For a partition $\lambda$, we define $l(\lambda)$ to be the length of the longest column in $\lambda$. We say a partition $\lambda$ is of \textit{type $\A_{m-1}$} if the $l(\lambda)\le m-1$. For two partitions $\lambda$, $\mu$, we say \textit{$\mu$ is contained in $\lambda$ ($\mu\subset \lambda$)} if the Young diagram of shape $\mu$ is contained in the Young diagram of shape $\lambda$ when aligned with
respect to their top left corners. 

If $\mathcal{X}$ is a totally ordered alphabet, the set of \textit{semi-standard Young tableaux} of shape $\lambda$ in the alphabet $\mathcal{X}$ is the set of fillings of the Young diagram $\lambda$ with letters of $\mathcal{X}$ such that the entries are strictly increasing along each column, and weakly increasing along each row. A Young tableau is called \textit{standard} if its rows are strictly increasing.

Let $\op{SSYT}(\lambda)$ be the set of semi-standard Young tableaux of shape $\lambda$ with entries in the ordered alphabet $\mathcal{A}_{m} = \{1< \ldots < m\}$ and let $\T\in \op{SSYT}(\lambda)$. The \textit{word} $W(\T)$ of the tableau $\T$ is obtained from it by reading its entries columnwise from right to left. 

To a number/letter $w \in \{1, \ldots , m\}$  we assign the path 

\begin{align*}
\pi_{w} : [0,1] & \rightarrow \mathfrak{h}^{*}_{\mathbb{R}} \\
t & \mapsto t \cdot \varepsilon_{w},
\end{align*}

Following \cite[Section 2.2]{branchconj}, to each semi-standard Young tableau $\T \in \op{SSYT}(\lambda)$ there is an associated path $\pi_{W(\T)}: [0,1] \rightarrow \mathfrak{h}^{*}_{\mathbb{R}}$ given for $W(T)=w_{1}\cdots w_{r}$ by the concatenation 
\noindent
\begin{equation*}\label{eq:pathword}
\pi_{W(T)}: = \pi_{w_{1}} * \cdots * \pi_{w_{r}},
\end{equation*}

where

\begin{equation*}
\pi_{1}*\pi_{2} (t) := \begin{cases} \pi_{1}(2t) &\text{ if } 0\leq t \leq \frac{1}{2} \\ \pi_{1}(1)+\pi_{2}(2t-1) &\text{ if } \frac{1}{2} \leq t \leq 1 \end{cases}.
\end{equation*}

\begin{center}
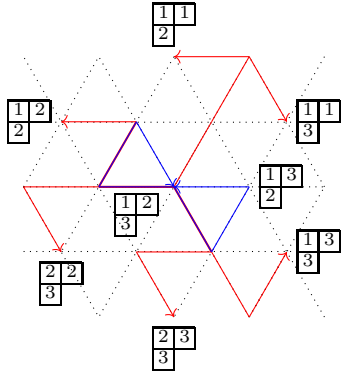
\begin{figure}[h!]\label{pathtabl}
\begin{tikzpicture}
\draw[-, dotted] (-2,0) --(2,0);
\draw[-, dotted] (-2,0.86) --(2,0.86);
\draw[-, dotted] (-2,-0.86) --(2,-0.86);
\draw[-, dotted] (0,0) --(0.5, 0.86602540378);
\draw[-,dotted] (0,0) --(-1, 0);
\draw[-,dotted] (-2,0) --(-1, -1.73205080756);
\draw[-,dotted] (-2,0) --(-1, 1.73205080756);
\draw[-,dotted] (2,0) --(1, 1.73205080756);
\draw[-,dotted] (1,-1.73205080756) --(-1, 1.73205080756);
\draw[-,dotted] (2,-1.73205080756) --(0, 1.73205080756);
\draw[-,dotted] (0,-1.73205080756) --(-2, 1.73205080756);
\draw[-,dotted] (1,-1.73205080756) -- (1.5,-0.86602540378);
\draw[-, dotted] (0,0) --(0.5, -0.86602540378);
\draw[-, dotted] (1,1.73205080756) --(-1, -1.73205080756);
\draw[-, dotted] (2,1.73205080756) --(0, -1.73205080756);
\draw[-, dotted] (1,-1.73205080756) -- (2,0);
\draw[-, dotted] (0,1.73205080756) -- (-1.5,-0.86602540378);
\draw[-, red] (0,0) --(0.5, 0.86602540378);
\draw[-,red] (0.5, 0.86602540378) -- (1, 1.73205080756);
\draw[->, red] (1, 1.73205080756) -- (0,1.73205080756);
\draw[->, red] (1, 1.73205080756) -- (1.5, 0.86602540378);
\draw[-,red, thick](0,0)--(-1,0);
\draw[-,red,thick](-1,0) -- (-0.5, 0.86602540378);
\draw[->,red]  (-0.5, 0.86602540378)--(-1.5,0.86602540378);
\draw[-,red, thick](0,0)--(0.5,-0.86602540378);
\draw[-, red ] (0.5,-0.86602540378) -- (1,-1.73205080756);
\draw[->,red] (1,-1.73205080756) -- (1.5,-0.86602540378);
\draw[-,red](-1, 0) -- (-2,0);
\draw[->,red] (-2,0) -- (-1.5, -0.86602540378);
\draw[-,red] (0.5, -0.86602540378) -- (-0.5, -0.86602540378);
\draw[->,red] (-0.5, -0.86602540378) -- (0, -1.73205080756);
\draw[-, blue] (0,0) -- (-1,0);
\draw[-,blue] (-1,0) -- (-0.5, 0.86602540378);
\draw[->,blue] (-0.5, 0.86602540378) -- (0,0);
\draw[-,blue] (0,0) -- (0.5, -0.86602540378);
\draw[-,blue] (0.5, -0.86602540378) -- (1,0);
\draw[->,blue] (1,0) -- (0,0);
\node[above] at (0,1.73205080756){\tiny{\Skew(0:1,1|0:2) }};
\node[right] at (1.5, 0.86602540378){\tiny{\Skew(0:1,1|0:3)}};
\node[left] at (-1.5, 0.86602540378){\tiny{\Skew(0:1,2|0:2)}};
\node[below] at (-1.5, -0.86602540378){\tiny{\Skew(0:2,2|0:3)}};
\node[below] at (0, -1.73205080756){\tiny{\Skew(0:2,3|0:3)}};
\node[right] at (1.5,-0.86602540378){\tiny{\Skew(0:1,3|0:3)}};
\node[above] at (-0.5,-0.8){\tiny{\Skew(0:1,2|0:3)}};
\node[right] at (1,0){\tiny{\Skew(0:1,3|0:2)}};
\end{tikzpicture}
\caption{Semistandard tableaux of shape $\lambda = (2,1)$ and their paths, for $\slth$.  }
\end{figure}
\end{center}

Now let $m = 2n$ for some $n \in  \mathbb{Z}_{\ge 1}$ and consider the automorphism $\sigma$ of $\mathfrak{sl}(2n,\mathbb{C})$ induced by the folding of the Dynkin diagram of type $\A_{2n-1}$ along the middle vertex.
The set of $\sigma$-fixed points $\mathfrak{sl}(2n,\mathbb{C})^{\sigma}$ is a sub Lie algebra isomorphic to $\mathfrak{sp}(2n,\mathbb{C})$. For an integral weight $\mu=\sum_{i=1}^{2n}a_i\varepsilon_i$ of $\mathfrak{sl}(2n,\mathbb{C})$ we define the weight $$\op{res}{\mu}:=\mu|_{\mathfrak{h}^{\sigma}_{\mathbb{R}}}=\sum_{i=1}^{n}a_i\varepsilon_i - \sum_{i=1}^{n}a_{2n-i+1}\varepsilon_{2n-i+1}$$ which is an integral weight of $\mathfrak{sp}(2n,\mathbb{C})$ with choice of simple roots given by $\{\op{res}(\alpha_i)\mid i \in \{1,\ldots, 2n-1\}\}$.

Let $\p_{SSYT}(\lambda)$ be the Littelmann path model for the simple $\mathfrak{sl}(2n,\mathbb{C})$ representation $\li(\lambda)$ of highest weight $\lambda$ which consists, by definition, of those paths associated to the set of semi-standard Young tableaux in $\op{SSYT}(\lambda)$. Each path $\pi:[0,1] \rightarrow \mathfrak{h}^{*}_{\mathbb{R}}$, may be restricted to a path
$\op{res}(\pi): [0,1] \rightarrow (\mathfrak{h}^{\sigma}_{\mathbb{R}})^{*}$ via $\op{res}(\pi)(t):= \pi(t)|_{\mathfrak{h}^{\sigma}_{\mathbb{R}}}$. The set $\op{domres}(\lambda)$ consists of restricted paths in $\op{res}(\p_{SSYT}(\lambda))$ that are contained in the dominant Weyl chamber of $\mathfrak{sl}(2n,\mathbb{C})^{\sigma}$, i.e. in the positive real span of $\{\op{res}(\omega_1),\ldots, \op{res}(\omega_n)\}$ which are the fundamental weights of $\mathfrak{sp}(2n,\mathbb{C})$ for our choice of simple roots.

\section{The set domres($\lambda$)}
Let $\op{SSYT}_{\mathcal{C}_{n}}$ be the set of all semi-standard Young tableaux of shapes the partitions of type $\A_{2n-1}$ with entries in the ordered alphabet 
$$\mathcal{C}_{n}=\{1<\ldots <n< \overline{n}<\ldots < \overline{1}\}.$$
The word $W(\T)$ of a semi-standard Young tableau $\T$ of this type is obtained from it as in Section \ref{wordsandpaths}. Also, to each word $W = w_{1}\cdots w_{r}$ is attached an integral weight of $\mathfrak{sp}(2n,\mathbb{C})$

$$\mu_{W} = \sum_{i=1}^{r}\varepsilon_{w_{i}}$$

\noindent where $\varepsilon_{\bar i} = -\varepsilon_{i}$. To $T \in \op{SSYT}_{\mathcal{C}_{n}}$ is associated a path $\pi_{W(\T)}$ by \eqref{eq:pathword}, where we set $\pi_{\bar i}=-\pi_i$.

\begin{defn}\label{dompropdef} A semi-standard Young tableau $\T\in \op{SSYT}_{\mathcal{C}_{n}}$ has the \textit{dominance property} if its associated path $\pi_{W(\T)}: [0,1] \rightarrow (\mathfrak{h}^{\sigma}_{\mathbb{R}})^{*}$ is dominant, i.e. it is contained in the dominant Weyl chamber of $\mathfrak{sp}(2n,\mathbb{C})$ (cf. \cite{jaz}).
\end{defn} 
Combinatorially one may break this down as follows. Let $k$ be the length of the word $W(\T)$. Then, reading the word from left to right define sub-words $W_{1}(\T), \ldots, W_{k}(\T) = W(\T)$ by adding one letter at a time. For example, for the word $W(\T) = 321\bar 3$, which has length $k = 4$, we get the following sequence of sub-words: $W_{1}(\T) = 3, W_{2}(\T) = 32, W_{3}(\T) = 321, W_{4}(\T) = 321 \bar 3 = W(\T)$. The semi-standard Young tableau $\T$ has the dominance property if and only if the weights $\mu_{W_{1}}, \ldots, \mu_{W_{k}}$ are all dominant. 

\begin{defn}\label{stable} Let $\mu$ be a partition of type $\A_{2n-1}$. In the case of $l(\mu)\le n$, we call $\mu$ \textit{stable}. 
\end{defn}

Let $\mu$ be a stable partition of type $\A_{2n-1}$. Note that $\op{res}(\mu)$ is a dominant weight for $\mathfrak{sp}(2n,\mathbb{C})$ which corresponds to the same partition as $\mu$.

\begin{defn}
\label{definitionofthesetdomres}
Let $\mu \subset \lambda$ be partitions of type $\A_{2n-1}$ such that $\mu$ is stable. We denote the set of $T \in \op{SSYT}_{\mathcal{C}_{n}}$ of shape $\lambda$ and weight $\mu_{W(\T)} = \mu$ that have the dominance property by $\op{domres}(\lambda, \mu)$.  
\end{defn}

The following fact is made clear in \cite{jaz}, Section 3.

\begin{fact}\label{fact}
The set $\op{domres(\lambda, \mu)}$ corresponds to paths in $\op{domres}(\lambda)$ with endpoint $\op{res}(\mu)$. We abuse notation and do not distinguish between tableaux and paths:

\begin{align*}
\op{domres}(\lambda, \mu) = \left\{\delta \in \op{domres}(\lambda): \delta(1) = \op{res}(\mu)\right\}.
\end{align*}

\end{fact}

\begin{ex}
Let $n = 2$ and $\lambda = \omega_{1}+ \omega_{2}$. Then $\op{domres}(\lambda, \mu) = \emptyset$ unless $\mu = \omega_{1} + \omega_{2}$, in which case we have 

$$\op{domres} = \left\{ \Skew(0: \hbox{\tiny{1}}, \hbox{\tiny{1}} | 0: \hbox{\tiny{2}}) \right\}$$

 or $\mu = \omega_{1}$, in which case:
 
 $$\op{domres} = \left\{ \Skew(0: \hbox{\tiny{1}}, \hbox{\tiny{1}} | 0: \hbox{\tiny{$\overline{1}$}}) \right\}$$

\end{ex}

By fact \ref{fact}, Theorem \ref{mainbranchingthm} reads in terms of tableaux as follows.

\begin{thm}\label{mainbranchingthmtab}

The following decomposition holds:

\begin{align*}
\op{res}^{\mathfrak{sl}(2n,\mathbb{C})}_{\mathfrak{sl}(2n,\mathbb{C})^{\sigma}} (\li(\lambda)) = \underset{\T \in \op{domres}(\lambda, \mu)}{\bigoplus} \tilde \li(\op{res}(\mu_{W(\T)}))
\end{align*}

\noindent where $ \tilde \li(\op{res}(\mu_{W(\T)}))$ denotes the simple module for $\mathfrak{sl}(2n,\mathbb{C})^{\sigma}$ of highest weight $\op{res}(\mu_{W(\T)})$, and $\op{res}^{\mathfrak{sl}(2n,\mathbb{C})}_{\mathfrak{sl}(2n,\mathbb{C})^{\sigma}} (\li(\lambda))$ denotes the restriction of $\li(\lambda)$ to $\mathfrak{sl}(2n,\mathbb{C})^{\sigma}$.
\end{thm}

\section{Littlewood-Richardson Sundaram tableaux}
\label{skewsection}
\begin{defn}
Let $\mu \subset \lambda $ be two partitions of type $\A_{2n-1}$. A \textit{ skew tableau $\mathscr{T}$ of skew shape $\lambda / \mu$} is a filling of a Young diagram of shape $\lambda$ leaving the boxes that belong to $\mu \subset \lambda$ blank, with the others having entries in the alphabet $\mathcal{A}_{2n}$, and such that these entries are strictly increasing along the columns and weakly increasing along the row. The word $W(\mathscr{T})$ of $\mathscr{T}$ is obtained just as for semi-standard Young tableaux, reading from right to left and from top to bottom, ignoring the blank boxes. 
\end{defn}

\begin{defn}\label{evdef}
A partition of type $\A_{2n-1}$ is \textit{even} if every column in its corresponding Young diagram has an even number of boxes.
\end{defn}

\begin{ex}
The semi-standard Young tableau $\Skew(0:\mbox{\tiny{1}},\mbox{\tiny{3}}|0:\mbox{\tiny{6}},\mbox{\tiny{10}})$ has shape the even partition $(2,2)$.
\end{ex}

\begin{defn}\label{LRS-tableau-def}
Let $\lambda , \mu, \eta $ be partitions of type $\A_{2n-1}$ such that $\mu \subset \lambda$. A \textit{Littlewood-Richardson} tableau of skew shape $\lambda / \mu$ and weight
$\eta$ is a skew tableau of skew shape $\lambda / \mu$ which has a dominant word of weight $\eta$. The set of all such skew tableaux is denoted by $\op{LR}(\lambda/ \mu, \eta)$.

A Littlewood-Richardson tableau of skew shape $\lambda / \mu$ and weight $\eta$ is called \textit{n-symplectic Sundaram} or just \textit{Sundaram} if $\eta$ is even, $\mu \subset \lambda$ is stable and $2i+1 $ does not appear strictly below row $n+i$ for $i \in \{0, 1, \ldots, \frac{1}{2} l(\eta)\}$). The set of all such tableaux is denoted by $\op{LRS}(\lambda/\mu, \eta)$.
\end{defn}

\begin{ex}
\label{firstexamplensymplectic}
The tableau $\mathscr{L} = \Skew(0:\mbox{},\mbox{\tiny{1}},\mbox{\tiny{1}}|0:\mbox{ },\mbox{\tiny{2}}|0:\mbox{\tiny{2}})$ is  a Littlewood-Richardson tableau of skew shape $\lambda / \mu$ and weight $\eta$ for $\lambda = \omega_{1} + \omega_{2} + \omega_{3}, \mu = \omega_{2},$ and $\eta = 2 \omega_{2}$ and the tableau $\mathscr{T} = \Skew(0: \mbox{ }|0: \mbox{ } |0: \mbox{\tiny{1}})$ is a Littlewood-Richardson tableau of skew shape $\lambda'/ \mu'$ and weight $\eta'$ for $\lambda' = \omega_{3}, \mu' = \omega_{2},$ and $\eta' = \omega_{1}$. Notice that  $\mathscr{L}$ is 2-symplectic Sundaram while $\mathscr{T}$ is not. 
\end{ex}

\begin{defn}\label{LRcoff}
The Littlewood-Richardson coefficient is defined as the number $c^{\lambda}_{\mu, \eta} \in \mathbb{Z}_{\geq 0}$ such that
\begin{align*}
\li(\mu) \otimes \li(\eta) = \underset{\mu \subset \lambda}{\bigoplus} c^{\lambda}_{\mu, \eta} \li(\lambda)
\end{align*}
\noindent
where $\li(\lambda),$ $\li(\mu),$ and $\li(\eta)$ are the corresponding simple representations of \linebreak $\mathfrak{sl}(2n, \mathbb{C})$. 
\end{defn}

\noindent
Theorem \ref{litrich} below is known as the \textit{Littlewood-Richardson rule}. It was first stated in 1934 by Littlewood and Richardson (see \cite{lr}). it was fully proven in the late 1970's by \cite{Sch},
\cite{T1}, \cite{T2}, \cite{Ma}.

\begin{thm}\cite{howelee}
\label{litrich}
The Littlewood-Richardson coefficients are obtained by counting Littlewood-Richardson tableaux: 
\begin{align*}
c^{\lambda}_{\mu, \eta} = |\op{LR}(\lambda/ \mu, \eta)|.
\end{align*}
\end{thm}

\begin{rem}
\label{redundance}
Definition \ref{LRcoff} implies that $c^{\lambda}_{\mu, \eta} = c^{\lambda}_{\eta, \mu}$. 
\end{rem}

We use the notation $c^{\lambda}_{\mu, \eta}(\s) = |\op{LRS(\lambda/ \mu, \eta)}|$. The following theorem was proven by Sundaram in Chapter IV of her PhD thesis \cite{sundaram}. See also Corollary 3.2 of \cite{sundarampaper}. For stable weights (i.e. with columns of at most length $n$) it was proven by Littlewood in \cite{litt} and is known as the Littlewood branching rule. 

\begin{thm} \cite{sundaram}
\label{sundaramstheorem}
Let $\lambda \in \mathfrak{h}^{*}_{\mathbb{R}}$ be a dominant integral weight. Then 

\begin{align*}
\op{res}^{\mathfrak{sl}(2n, \mathbb{C})}_{\mathfrak{sl}(2n, \mathbb{C})^{\sigma}}(\li(\lambda)) = \underset{l(\mu) \leq n}{\underset{\mu \subset \lambda}{\bigoplus}} \N_{\lambda, \mu} \tilde{\li}(\op{res}(\mu))
\end{align*}

\noindent
where 
\begin{align*}
\N_{\lambda, \mu} = \underset{\eta \hbox{ \tiny{even} }}{\sum}c^{\lambda}_{\mu, \eta}(\s). 
\end{align*}
Recall that $\op{res}(\mu) = \mu|_{\mathfrak{h}^{\sigma}_{\mathbb{R}}}$ and that since $l(\mu) \leq n$, $\op{res}(\mu)$ corresponds to the partition $\mu$. 
\end{thm}

\section{Symplectic RSK correspondence}
\label{arrays}

For the comfort of the reader we recall some facts about the combinatorics of two-line arrays which we need in the next section. We start with the definition of (column) bumping. Let $\T$ be a semi-standard tableau in any totally ordered alphabet, and $l$ a letter in this alphabet. A new tableaux, denoted by $l \rightarrow \T$ is obtained by column bumping $l$ into $\T$ as follows. If all the entries in the first column of $\T$ are smaller or equal to $l$, place $l$ at the bottom of this column. This is the new tableau. Otherwise, replace by $l$ the smallest entry which is greater than $l$, let us call this entry $l'$. Now consider the second column of $\T$, and proceed with $l'$ as with $l$ and the first column. If there are no more columns available, create a new one. Now let $d$ b an entry in a semi-standard tableau $\T$ which is an inner corner of its Young diagram, i.e. it is an entry in the last box in one of its rows. A new tableau $\T'$ is created by (column) bumping out  $d$ from $\T$ as follows. In the column previous to the one where $d$ is, find the largest entry $d'$ that is smaller than or equal to $d$, and replace it with $d$. Now look at the next column, and proceed with $d'$ as with $d$ at the beginning of the procedure. In the end, some entry $d^{end}$  from the first column will have been removed from $\T$ The tableau $\T'$ has the property that $d^{end}\rightarrow(\T') = \T$, i.e., column bumping $d^{end}$ into $\T'$ will give back the initial tableau $\T$.

\subsection{Two-line arrays and even partitions}

\begin{defn}
\label{twolinearray}
A \textit{special two-line array} $\L$ is a two-line array of pairwise distinct positive integers
\begin{align*}
\L=\begin{bmatrix}
j_{1} & \cdots & j_{r} \\
i_{1} & \cdots & i_{r}
\end{bmatrix}
\end{align*} 

\noindent
such that 

\begin{itemize}
\label{1ta}
\item[1.] $j_{1} < \ldots < j_{r}$
\label{2ta}
\item[2.] $j_{s} > i_{s}, s \in \{1, \ldots, r\}.$
\end{itemize}

\end{defn}

Consider a special two-line array $\L$ as above and let $s_{1}, \ldots, s_{r}$ be the re-ordering of the index set $\{1, \ldots, r\}$ such that $i_{s_{1}}< \ldots < i_{s_{r}}$. We can obtain a standard Young tableau $\E(\L)$ from our array by column bumping its entries: 

\begin{equation*}\label{arraybump}
\E(\L):=j_{s_{1}} \rightarrow \cdots \rightarrow j_{s_{r}} \rightarrow i_{1} \rightarrow \cdots \rightarrow i_{r} \rightarrow \emptyset .
\end{equation*}

\begin{ex}
\label{bumping}
Consider the two-line array 
$$\L=\begin{bmatrix}6 & 10 \\ 1 & 3 \end{bmatrix}$$

\noindent
We have

\begin{align*} \E_{\L} = 
6 \rightarrow 10 \rightarrow 1 \rightarrow 3 \rightarrow \emptyset =\Skew(0:\mbox{\tiny{1}},\mbox{\tiny{3}}|0:\mbox{\tiny{6}},\mbox{\tiny{10}}).
\end{align*}

\end{ex}

Recall the definition of an even partition from Definition \ref{evdef}. The following theorem follows from the Burge correspondence \cite{burge} (see also Theorem 3.31 in \cite{sundaram}) and Lemma 10.7 in Sundaram's thesis \cite{sundaram}.
\begin{thm}\cite{sundaram}\label{even}
\label{burge}
The assignment $\L \mapsto \E_{\L}$ above defines a bijection between special two-line arrays and standard Young tableaux of even shape.
\end{thm}

\subsection{Up-down tableaux and Q-symbols}
We now recall the definition of up-down sequences of partitions, which, for us, replace words in the classical RSK correspondence. 

\begin{defn} An \textit{up-down tableau} of length $k$ is a $k$-sequence of partitions
$$S^k_{\mu}=(\emptyset=\mu_0,\mu_1,\ldots,\mu_k=\mu)$$
of type $\A_{n}$ such that $\mu_j$ and $\mu_{j+1}$ differ by exactly one box for all $j\in \{0,1,\ldots,k-1\}$. We call $\mu$ the shape of the up-down sequence $S^k_{\mu}$.
\end{defn}

Note that for every up-down sequence $S^k_{\mu}$, the partition $\mu_1$ consists of exactly one box.

\begin{defn} \label{partialq}Let $S=S^k_{\mu}$ be an up-down sequence. We associate to $S$ a sequence of standard tableaux which are fillings of $\mu_1,\ldots,\mu_k$ with entries in the alphabet
$$\{1 < 2 < \ldots < k\}$$
successively as follows. We fill the unique box of $\mu_1$ with the entry $1$ to obtain a tableau $\Q_{S}^p(1)$. Assume that the Young diagrams $\mu_1,\mu_2,\ldots,\mu_j$ have already been given a filling such that we have a sequence of tableaux $\Q_{S}^p(1),\ldots, \Q_{S}^p(j)$. If $\mu_{j+1}$ is obtained from $\mu_j$ by adding a box, $\Q_{S}^p(j+1)$ is obtained from $\Q_{S}^p(j)$ by filling this box with the letter $``j+1"$. If $\mu_{j+1}$ is obtained from $\mu_{j}$ by removing a box, $\Q_{S}^p(j+1)$ is obtained from $\Q_{S}^p(j)$ by column-bumping the entry lying in this box out of $\Q_{S}^p(j)$. In other words $\Q_{S}^p(j+1)$ is obtained from $\Q_{S}^p(j)$ by removing the left-most (or smallest) entry (which we call $r_{j}$) from the row in which the removed box lies, and shifting everything else to the left. 

We call the tableau $\Q_{S}^p(j)$ the \textit{partial $\Q$-symbol of $S$ at step $j$}. The partial $\Q$-symbol of $S$ at step $k$ is called the \textit{partial $\Q$-symbol of $S$} and is denoted by $\Q^p_S$.
\end{defn}

\begin{ex}\label{run} 
Let $k=10, n\geq 3$, $\mu= \Skew(0:, ,|0: , |0: )$ and $S=S^k_{\mu}=$
\begin{align*} & \Biggl(  \emptyset,\quad  \Skew(0:) ,\quad \Skew(0:,) , \quad \Skew(0:, |0:) , \quad \Skew(0:, ,|0:) , \Skew(0:, ,|0: , ),  \\ 
 &   \quad \Skew(0:, |0: , ) , \quad \Skew(0:, ,|0: , ) , \quad \Skew(0:, ,|0: , ,) , \quad \Skew(0:, ,|0: , , |0: )  , \quad  \Skew(0:, ,|0: , |0: ) \quad \Biggr).
\end{align*}
Table 1 shows the corresponding partial $\Q$-symbols at each step. 

\begin{table}[!htb]
\caption{}
    \begin{minipage}{.5\linewidth}

      \centering
\begin{tabular}{ c | c  }

$    j$ & $    \Q_{S}^p(j)$ \\ \hline
& \\
$    1$ & $    \Skew(0:\mbox{\tiny{1}})$\\ & \\
$    2$ & $    \Skew(0:\mbox{\tiny{1}},\mbox{\tiny{2}})$ \\ & \\
$    3$ & $    \Skew(0:\mbox{\tiny{1}},\mbox{\tiny{2}},\mbox{\tiny{3}})$ \\ & \\ 
$    4$ & $    \Skew(0:\mbox{\tiny{1}},\mbox{\tiny{2}},\mbox{\tiny{4}}|0:\mbox{\tiny{3}})$ \\ & \\ 
$    5$ &  $    \Skew(0:\mbox{\tiny{1}},\mbox{\tiny{2}},\mbox{\tiny{4}}|0:\mbox{\tiny{3}},\mbox{\tiny{5}})$ \\ & \\
$    6$ &  $    \Skew(0:\mbox{\tiny{2}},\mbox{\tiny{4}}|0:\mbox{\tiny{3}},\mbox{\tiny{5}}) $ \\ & \\ 
$    7$ & $    \Skew(0:\mbox{\tiny{2}},\mbox{\tiny{4}}, \mbox{\tiny{7}}|0:\mbox{\tiny{3}},\mbox{\tiny{5}}) $ \\ & \\ 
$    8 $& $    \Skew(0:\mbox{\tiny{2}},\mbox{\tiny{4}}, \mbox{\tiny{7}}|0:\mbox{\tiny{3}},\mbox{\tiny{5}},\mbox{\tiny{8}})$ \\ & \\ 
$    9$ & $    \Skew(0:\mbox{\tiny{2}},\mbox{\tiny{4}}, \mbox{\tiny{7}}|0:\mbox{\tiny{3}},\mbox{\tiny{5}},\mbox{\tiny{8}}|0:\mbox{\tiny{9}})$ \\ & \\ 
$    10$ & $    \Skew(0:\mbox{\tiny{2}},\mbox{\tiny{4}}, \mbox{\tiny{7}}|0:\mbox{\tiny{5}},\mbox{\tiny{8}}|0:\mbox{\tiny{9}})$ .   \\ 
\end{tabular}
    \end{minipage} 
\end{table}
\end{ex}

\begin{defn}\label{twolinarrdef} Let $S=S^k_{\mu}$ be an up-down sequence. We associate to $S$ an even partition as follows. Every time $\mu_{j+1}$ is obtained from $\mu_j$ by removing a box, we save the step $j$ together with the entry $r_{j}$ of $\Q_{S}^p(j)$ in a special two-line array as $\begin{bmatrix}j \\ r_{j} \end{bmatrix}$ and concatenate the two-line arrays obtained this way, with the first corresponding to the smallest $j$ we saved. In the end we get a special two-line array which we denote by $\L(S)$. We denote the even partition $\E_{\L(S)}$ obtained by Theorem \ref{even} by $\E_{S}.$ With notation as in Section \ref{arrays} for $\L(S)$, we have by \eqref{arraybump} that: 
 
 \begin{align*}
 \E_{S} = j_{s_{1}} \rightarrow \cdots \rightarrow j_{s_{r}} \rightarrow i_{1} \rightarrow \cdots \rightarrow i_{r} \rightarrow \emptyset .
 \end{align*}

\end{defn}

\begin{ex}\label{run2} In Example \ref{run} there are two steps where a box is removed, namely step $6$ where the entry $1$ is removed from $\Q_{S}^p(5)$ to obtain $\Q_{S}^p(6)$ and in step $10$ where the entry $3$ is removed from $\Q_{S}^p(9)$ to obtain $\Q_{S}^p(10)$.

Hence, we concatenate $\begin{bmatrix}6 \\ 1 \end{bmatrix}$ with $\begin{bmatrix} 10 \\ 3 \end{bmatrix}$  to deduce 
$$\L(S)= \begin{bmatrix} 6 & 10 \\ 1 & 3 \end{bmatrix}.$$

Using \eqref{arraybump}, we obtain (compare with Example \ref{bumping})
$$ \E_{S}=\Skew(0:\mbox{\tiny{1}},\mbox{\tiny{3}}|0:\mbox{\tiny{6}},\mbox{\tiny{10}}) \ . $$

\end{ex}

\begin{defn}\label{finalqdef} Let $S=S^k_{\mu}$ be an up-down sequence. We associate a standard Young tableau $\Q_S$ to $S$ by column bumping the entries of $\E_{S}$ into $\Q^{p}_{S}$ as follows:

\begin{align*}
\Q_{S} :=  j_{s_{1}} \rightarrow \cdots \rightarrow j_{s_{r}} \rightarrow i_{1} \rightarrow \cdots \rightarrow i_{r} \rightarrow \Q^{p}_{S} . 
\end{align*}
We call $\Q_{S}$ the \textit{final $Q$-symbol of $S$}.
\end{defn}

\begin{ex}\label{run3} For $S$ as in Example \ref{run} we have, using the calculations of Examples \ref{run} and  \ref{run2}:

$$\Q_S =6 \rightarrow 10 \rightarrow 1 \rightarrow 3 \rightarrow \Skew(0:\mbox{\tiny{2}},\mbox{\tiny{4}}, \mbox{\tiny{7}}|0:\mbox{\tiny{5}},\mbox{\tiny{8}}|0:\mbox{\tiny{9}})= \Skew(0:\mbox{\tiny{1}},\mbox{\tiny{2}}, \mbox{\tiny{4}}, \mbox{\tiny{7}}|0:\mbox{\tiny{3}},\mbox{\tiny{5}}, \mbox{\tiny{8}}|0: \mbox{\tiny{6}}, \mbox{\tiny{9}}|0: \mbox{\tiny{10}}).$$
\end{ex}

\begin{defn}\label{phi1} Let $S=S^k_{\mu}$ be an up-down sequence and $\lambda$ the shape of $\Q_S$ and $\nu$ the shape of $\E_{S}$. We associate a skew tableau $\tilde{\phi}(S)$ of skew shape $\lambda/ \mu$ and weight $\nu$ as follows. For each entry $j$ in $\E_{S}$, let $r(j)$ be the row to which it belongs (in $\E_{S}$). Write this number in the skew shape $\lambda / \mu$ in the row of $\Q_{S}$ where $j$ lies. 
\end{defn}

\begin{ex}\label{run4} For $S$ as in Example \ref{run} we have, using the calculations of Examples \ref{run}, \ref{run2}, \ref{run3}, that
$$\tilde{\phi}(S)=\Skew(0: , , , \mbox{\tiny{1}}|0: , , \mbox{\tiny{1}}|0: , \mbox{\tiny{2}}|0: \mbox{\tiny{2}}).$$ 
\end{ex}

Note that the tableau $\tilde{\phi}(S)$ produced in Example \ref{run4} is indeed a  Littlewood Richardson Sundaram tableau $\tilde{\phi}(S) \in \op{LRS}(\lambda/ \mu, \eta)$. This is always the case as shown in the proof of Theorem 8.11 in \cite{sundaram} (and Theorem 9.4). In fact, the following theorem holds (it is stated and proven in Theorems 8.14 and 9.4 of \cite{sundaram}). 

\begin{thm}
\label{symplecticrsk}
The correspondence 
$$S \longleftrightarrow (\Q_S,\tilde{\phi}(S))$$
is a bijection between up-down tableaux $S=S^k_{\mu}$
of length $k$ and shape $\mu$ and pairs $(\Q, \mathscr{L})$, where $\Q$ is a standard Young tableau of shape $\lambda$ with entries precisely the elements of the set $\{1, \ldots, k\}$, and $\mathscr{L}$ is a Littlewood-Richardson Sundaram tableau of skew shape $\lambda / \mu$ and even weight $\eta$ which is the shape of $\E_S$.
\end{thm}

\section{The bijection}

By Theorem \ref{sundaramstheorem}, a proof of Theorem \ref{mainbranchingthmtab} (and, equivalenty, a proof of Theorem \ref{mainbranchingthm}) would be established by the existence of a bijection

\begin{align}
\label{thebijection}
\op{domres}(\lambda, \mu) \overset{1:1}{\longleftrightarrow} \underset{\eta \hbox{ \tiny{even} }}{\underset{\mu \subseteq \lambda;}{\bigcup}} \op{LRS}(\lambda/\mu, \eta).
\end{align}
\noindent
 
\begin{defn}
\label{updownsequence}
Let $\T \in \op{domres}(\lambda, \mu)$. We associate an up-down sequence $S(\T)$ of weight $\mu$ to $\T$ as the sequence of weights of the vertices of the path $\pi_{W(\T)}$ corresponding to it. 

In other words, consider the word $W(\T)$ and start reading it from left to right. We produce a sequence of partitions of length the length of $W(\T)$ as we read $W(\T)$ by adding a box in row $i$ whenever there is an $i$ in $W(\T)$ and by removing a box from row $i$ whenever there is an $\bar i$. The last partition in the sequence is the partition associated to the dominant weight $\mu$.  
\end{defn}

\begin{ex}
\label{sequenceofshapes}
In Example \ref{guidingexample} the up-down sequence associated to each \linebreak $\T_{i} \in \op{domres}(\lambda,\mu)$ is displayed, for $\lambda = \omega_{1}+\omega_{2}+\omega_{3}+\omega_{4}$ and $\mu = \omega_{1}+\omega_{2}+\omega_{3}$. 
\end{ex}

\begin{defn} 
\label{qlambda}
Let $\lambda$ be the Young diagram associated to a partition. We define the standard tableau $\Q_{\lambda}$ associated to $\lambda$ by numbering each box of $\lambda$ in the order of our word reading, and reordering each row so that the resulting tableau is in fact standard. The standard tableau $\Q_{\lambda}$ defined in this way is precisely the $\Q$-symbol associated to the word of any semi-standard tableau of shape $\lambda$ in the classical RSK correspondence. 
\end{defn}

\begin{ex}
\label{exampleqlambda} Let $\lambda = (4,3,2,1)$. Then 
\begin{align*}
\Q_{\lambda} = \Skew(0:\mbox{\tiny{1}},\mbox{\tiny{2}}, \mbox{\tiny{4}}, \mbox{\tiny{7}}|0:\mbox{\tiny{3}},\mbox{\tiny{5}}, \mbox{\tiny{8}}|0: \mbox{\tiny{6}}, \mbox{\tiny{9}}|0: \mbox{\tiny{10}}).
\end{align*}
\end{ex}

\begin{prop}
\label{importantprop}
Let $\T \in \op{domres}(\lambda, \mu)$. Then its final $\Q$-symbol $\Q_{S(\T)}$ is equal to $\Q_{\lambda}$. In particular, it has shape $\lambda$. Moreover, any two elements in $\op{domres}(\lambda, \mu)$ have the same partial $\Q$-symbol. 
\end{prop}

Consequently, 

\begin{equation*}
T \mapsto \tilde{\phi}(S(\T)),
\end{equation*}

with $\tilde{\phi}(S(\T))$ defined in Definition \ref{phi1}, is a map from $\op{domres}(\lambda, \mu)$ to $$\underset{\eta \hbox{ \tiny{even} }}{\underset{\mu \subseteq \lambda;}{\bigcup}} \op{LRS}(\lambda/\mu, \eta)$$ by Theorem \ref{symplecticrsk}. We prove that this is indeed a bijection establishing \eqref{thebijection}.

\begin{thm}\label{bijection} The map
\begin{align*} \phi: \op{domres}(\lambda, \mu) & \rightarrow \underset{\eta \hbox{ \tiny{even} }}{\underset{\mu \subseteq \lambda;}{\bigcup}} \op{LRS}(\lambda/\mu, \eta) \\
T & \mapsto \tilde{\phi}(S(T))
\end{align*}
is a bijection.
\end{thm}

Before proving Proposition \ref{importantprop} and Theorem \ref{bijection}, let us consider an example.

\begin{ex}
\label{guidingexample}
Let $n \geq 3, \lambda = \omega_{1}+\omega_{2}+\omega_{3}+\omega_{4}$ and $\mu = \omega_{1}+\omega_{2}+\omega_{3}$. Then $\op{domres}(\lambda, \mu)$ consists of the three elements $\T_{1}, \T_{2},$ and $\T_{3}$ below:

$${ \T_{1} = \Skew(0:\mbox{\tiny{1}},\mbox{\tiny{1}},\mbox{\tiny{1}},\mbox{\tiny{1}}|0:\mbox{\tiny{2}},\mbox{\tiny{2}},\mbox{\tiny{2}}|0:\mbox{\tiny{3}},\mbox{\tiny{$\overline{1}$}}|0:\mbox{\tiny{$\overline{2}$}})} \qquad \T_{2} = \Skew(0:\mbox{\tiny{1}},\mbox{\tiny{1}},\mbox{\tiny{1}},\mbox{\tiny{1}}|0:\mbox{\tiny{2}},\mbox{\tiny{2}},\mbox{\tiny{$\overline{1}$}}|0:\mbox{\tiny{3}}, \mbox{\tiny{3}}|0:\mbox{\tiny{$\overline{3}$}}) \qquad \T_{3} = \Skew(0:\mbox{\tiny{1}},\mbox{\tiny{1}},\mbox{\tiny{1}},\mbox{\tiny{1}}|0:\mbox{\tiny{2}},\mbox{\tiny{2}},\mbox{\tiny{2}}|0:\mbox{\tiny{3}}, \mbox{\tiny{$\overline{2}$}}|0:\mbox{\tiny{$\overline{1}$}}).$$

The associated up-down sequences $\op{S}(\T_{1}), \op{S}(\T_{2}), \op{S}(\T_{3})$ look, by Definition \ref{updownsequence}, as follows:
\begin{align*}
\op{S}(\T_{1})&= \left(
\Skew(0:), \quad \Skew(0:,), \quad \Skew(0:, |0:), \quad    \Skew(0:, ,|0:), \quad  \Skew(0:, ,|0: , ), \right.\\
 &\left.  \quad \Skew(0:, |0: , ),  \quad \Skew(0:, ,|0: , ), \quad   \Skew(0:, ,|0: , ,), \quad \Skew(0:, ,|0: , , |0: ), \quad  \Skew(0:, ,|0: , |0: )\right) \\
 \op{S}(\T_{2}) &=  
 \left(\Skew(0:) ,\quad \Skew(0:,) ,\quad \Skew(0:) ,\quad \Skew(0:,) ,\quad  \Skew(0:, |0:), \right. \\
& \left. \quad \Skew(0:, |0: |0: ) ,\quad \Skew(0:, , |0: |0: ) ,\quad \Skew(0:, ,|0: , |0:) ,\quad \Skew(0:, ,|0: , |0: , )  ,\quad  \Skew(0:, ,|0: , |0: ) \right) \\
 \op{S}(\T_{3}) &= 
\left( \Skew(0:) ,\quad \Skew(0:,) ,\quad \Skew(0:, |0:) ,\quad \Skew(0:, ,|0:) ,\quad \Skew(0:, ,|0: , ), \right. \\
& \left. \quad \Skew(0:, ,|0: ) ,\quad \Skew(0:, , ,|0:  ) ,\quad \Skew(0:, , ,|0: ,) ,\quad \Skew(0:, , ,|0: , |0: )  ,\quad  \Skew(0:, ,|0: , |0: )\right).
 \end{align*}

In Table 2 we give the steps of the word reading in the leftmost column and the respective partial $Q$-symbol (see Definition \ref{partialq}) at each of these steps in the other columns.

\begin{table}[!htb] \label{table}
\caption{ The partial $\Q$-symbols. }
\begin{minipage}{1\linewidth}

      \centering
\begin{tabular}{ c | c  | c  | c }

$    j$ & $    \Q_{S(T_1)}^p(j)$ & $    \Q_{S(T_2)}^p(j)$ & $    \Q_{S(T_3)}^p(j)$ \\ 
& & & \\
$1$ & $    \Skew(0:\mbox{\tiny{1}})$ & 
$    \Skew(0:\mbox{\tiny{1}})$ & $    \Skew(0:\mbox{\tiny{1}})$ \\ [0.3 in]  
$    2$ & $    \Skew(0:\mbox{\tiny{1}},\mbox{\tiny{2}})$ & $   \Skew(0:\mbox{\tiny{1}},\mbox{\tiny{2}})$ & $    \Skew(0:\mbox{\tiny{1}},\mbox{\tiny{2}})$  \\ [0.3 in]  
$    3$ & $    \Skew(0:\mbox{\tiny{1}},\mbox{\tiny{2}}|0:\mbox{\tiny{3}})$ & $\Skew(0:\mbox{\tiny{2}})$ & 
$    \Skew(0:\mbox{\tiny{1}},\mbox{\tiny{2}}|0:\mbox{\tiny{3}})$
\\  [0.3 in] 
$    4$ & $    \Skew(0:\mbox{\tiny{1}},\mbox{\tiny{2}},\mbox{\tiny{4}}|0:\mbox{\tiny{3}})$ & $    \Skew(0:\mbox{\tiny{2}},\mbox{\tiny{4}})$ & $   \Skew(0:\mbox{\tiny{1}},\mbox{\tiny{2}},\mbox{\tiny{4}}|0:\mbox{\tiny{3}})$ \\ [0.3in] 
$    5$ & $    \Skew(0:\mbox{\tiny{1}},\mbox{\tiny{2}},\mbox{\tiny{4}}|0:\mbox{\tiny{3}},\mbox{\tiny{5}})$ & $    \Skew(0:\mbox{\tiny{2}},\mbox{\tiny{4}} |0:\mbox{\tiny{5}})$ & $\Skew(0:\mbox{\tiny{1}},\mbox{\tiny{2}},\mbox{\tiny{4}}|0:\mbox{\tiny{3}},\mbox{\tiny{5}})$ \\  [0.3 in] 
$    6$ & $     \Skew(0:\mbox{\tiny{2}},\mbox{\tiny{4}}|0:\mbox{\tiny{3}},\mbox{\tiny{5}})$ & $    \Skew(0:\mbox{\tiny{2}},\mbox{\tiny{4}}|0:\mbox{\tiny{5}}|0:\mbox{\tiny{6}})$ & $    \Skew(0:\mbox{\tiny{1}},\mbox{\tiny{2}},\mbox{\tiny{4}}|0:\mbox{\tiny{5}})$ \\ [0.3 in] 
$    7$ & $    \Skew(0:\mbox{\tiny{2}},\mbox{\tiny{4}}, \mbox{\tiny{7}}|0:\mbox{\tiny{3}},\mbox{\tiny{5}})$ & $    \Skew(0:\mbox{\tiny{2}},\mbox{\tiny{4}}, \mbox{\tiny{7}}|0:\mbox{\tiny{5}}|0:\mbox{\tiny{6}})$ & $   \Skew(0:\mbox{\tiny{1}},\mbox{\tiny{2}},\mbox{\tiny{4}}, \mbox{\tiny{7}}|0:\mbox{\tiny{5}})$ \\ [0.3 in] 
$    8$ & $    \Skew(0:\mbox{\tiny{2}},\mbox{\tiny{4}}, \mbox{\tiny{7}}|0:\mbox{\tiny{3}},\mbox{\tiny{5}},\mbox{\tiny{8}})$ & $   \Skew(0:\mbox{\tiny{2}},\mbox{\tiny{4}}, \mbox{\tiny{7}}|0:\mbox{\tiny{5}},\mbox{\tiny{8}}|0:\mbox{\tiny{6}})$ & $   \Skew(0:\mbox{\tiny{1}},\mbox{\tiny{2}},\mbox{\tiny{4}}, \mbox{\tiny{7}}|0:\mbox{\tiny{5}},\mbox{\tiny{8}})$
\\ [0.3 in]  
$    9 $ & $\Skew(0:\mbox{\tiny{2}},\mbox{\tiny{4}}, \mbox{\tiny{7}}|0:\mbox{\tiny{3}},\mbox{\tiny{5}},\mbox{\tiny{8}}|0:\mbox{\tiny{9}})$ & $    \Skew(0:\mbox{\tiny{2}},\mbox{\tiny{4}}, \mbox{\tiny{7}}|0:\mbox{\tiny{5}},\mbox{\tiny{8}}|0:\mbox{\tiny{6}},\mbox{\tiny{9}})$ & $   \Skew(0:\mbox{\tiny{1}},\mbox{\tiny{2}},\mbox{\tiny{4}}, \mbox{\tiny{7}}|0:\mbox{\tiny{5}},\mbox{\tiny{8}}|0:\mbox{\tiny{9}})$
\\[0.3 in] 
$    10$ & $     \Skew(0:\mbox{\tiny{2}},\mbox{\tiny{4}}, \mbox{\tiny{7}}|0:\mbox{\tiny{5}},\mbox{\tiny{8}}|0:\mbox{\tiny{9}})$ & $    \Skew(0:\mbox{\tiny{2}},\mbox{\tiny{4}}, \mbox{\tiny{7}}|0:\mbox{\tiny{5}},\mbox{\tiny{8}}|0:\mbox{\tiny{9}})$ & $   \Skew(0:\mbox{\tiny{2}},\mbox{\tiny{4}}, \mbox{\tiny{7}}|0:\mbox{\tiny{5}},\mbox{\tiny{8}}|0:\mbox{\tiny{9}})$\\

%
%
%
\end{tabular}
 \\ [0.3 in]
\end{minipage}
\end{table}
Following Definition \ref{twolinarrdef}, every time a box is removed, we save the step $j$ together with the removed entry of the partial $Q$-symbol at the step before to obtain the special two-line arrays

$$\L_{\op{S}(\T_{1})} = \begin{bmatrix} 6 & 10 \\ 1 & 3\end{bmatrix} \qquad \L_{\op{S}(\T_{2})} = \begin{bmatrix} 3 & 10 \\ 1 & 6\end{bmatrix} \qquad \L_{\op{S}(\T_{3})} = \begin{bmatrix} 6 & 10 \\ 3 & 1\end{bmatrix}$$
and, by \eqref{arraybump}, the standard tableaux
$$\E_{\op{S}(\T_{1})} = \Skew(0:\mbox{\tiny{1}},\mbox{\tiny{3}}|0:\mbox{\tiny{6}},\mbox{\tiny{10}}) \qquad \E_{\op{S}(\T_{2})} = \Skew(0:\mbox{\tiny{1}},\mbox{\tiny{6}}|0:\mbox{\tiny{3}},\mbox{\tiny{10}}), \qquad \E_{\op{S}(\T_{3})} = \Skew(0:\mbox{\tiny{1}}|0:\mbox{\tiny{3}}|0: \mbox{\tiny{6}}|0: \mbox{\tiny{10}})\quad .$$
Following Definition \ref{finalqdef}, the associated final $Q$-symbols are
\normalsize
$$\Q_{\op{S}(\T_1)} = \Skew(0:\mbox{\tiny{1}},\mbox{\tiny{2}}, \mbox{\tiny{4}}, \mbox{\tiny{7}}|0:\mbox{\tiny{3}},\mbox{\tiny{5}}, \mbox{\tiny{8}}|0: \mbox{\tiny{6}}, \mbox{\tiny{9}}|0: \mbox{\tiny{10}}) \qquad \Q_{\op{S}(\T_2)} = \Skew(0:\mbox{\tiny{1}},\mbox{\tiny{2}}, \mbox{\tiny{4}}, \mbox{\tiny{7}}|0:\mbox{\tiny{3}},\mbox{\tiny{5}}, \mbox{\tiny{8}}|0: \mbox{\tiny{6}}, \mbox{\tiny{9}}|0: \mbox{\tiny{10}}) \qquad \Q_{\op{S}(\T_3)} = \Skew(0:\mbox{\tiny{1}},\mbox{\tiny{2}}, \mbox{\tiny{4}}, \mbox{\tiny{7}}|0:\mbox{\tiny{3}},\mbox{\tiny{5}}, \mbox{\tiny{8}}|0: \mbox{\tiny{6}}, \mbox{\tiny{9}}|0: \mbox{\tiny{10}}). $$
yielding 
\begin{align*}
\phi(\T_{1}) = \Skew(0: , , , \mbox{\tiny{1}}|0: , , \mbox{\tiny{1}}|0: , \mbox{\tiny{2}}|0: \mbox{\tiny{2}}) \qquad \phi(\T_{2}) = \Skew(0: , , , \mbox{\tiny{1}}|0: , , \mbox{\tiny{2}}|0: , \mbox{\tiny{1}}|0: \mbox{\tiny{2}}) \qquad \phi(\T_{3}) = \Skew(0: , , , \mbox{\tiny{1}}|0: , , \mbox{\tiny{2}}|0: , \mbox{\tiny{3}}|0: \mbox{\tiny{4}}).
\end{align*}
\end{ex}


\subsection{Proof of Proposition \ref{importantprop} and Theorem \ref{bijection}}

\begin{defn} Let us call a \textit{cancellation} a step $j$ such that the $j-th$ element of the word $W(\T)$ is a barred letter. Denote by $\lambda_{s}$ the largest stable partition contained in the Young diagram of shape $\lambda$: that is, the maximal sub-shape such that its columns have at most length $n$. 
\end{defn}

\begin{ex}
In Example \ref{guidingexample}, if we set $n = 3$, then 

$$\lambda_{s} = \Skew(0: , , , |0: , , |0: , ).$$
\end{ex}

\noindent
\begin{rem} \label{rinrowr}
All the boxes surrounding the area determined by $\lambda_{s}$ in a given element of $\op{domres}(\lambda, \mu)$ must be barred and therefore the steps at which they appear do not appear in the partial $\Q$-symbol. Also note that, by dominance, all entries occurring in row $r$ within $\lambda_{s}$ of an element of $\op{domres}(\lambda, \mu)$ must me equal to $r$.
\end{rem}

\begin{lem} Let $T,T'\in \op{domres}(\lambda,\mu)$. Then $\Q^p_{T}=\Q^p_{T'}$.

\end{lem}


\begin{proof}

It follows from Definition \ref{partialq} that it is enough to show that $\Q^{p}_{\T}$ and $\Q^{p}_{\T'}$ have the same entries. Since both partial $\Q$-symbols have the same shape 
$\mu$, a fixed number of boxes will be bumped out of each row $r$. Recall that in our case this amounts to removing the box at the beginning of that row and shifting the remaining entries of that row to the left. Therefore it suffices to show that if $j$ is a cancellation for $\T$, then $j$ does not appear in $\Q^{p}_{\T'}$. Otherwise, the $j$-th letter in $W(\T)$ is barred but the $j$-th letter in $W(\T')$ is not. Assume that this happens at row $r$. Then, by Remark \ref{rinrowr}, $r \leq n$ and the unbarred entry in $\T'$ must be an $r$. Moreover, to the right of the step $j$ entry of $\T$ in row $r$ can be no more $r$'s. This means that $\mu_{r} \leq s-1$, where $s$ is the column at which the $j$-th step appears. Moreover, by semistandardness, in row $r$ of $\T'$, all entries to the left of the $j$-th step entry must also be equal to $r$. This means that, in $\T'$, at least an $\bar r$ must appear to the southwest quadrant  pivoted by the $j$-th entry. This is true for every entry equal to $r$ in row $r$ and column $S \geq s$ of $\T$. These cancellations will eliminate the corresponding entries from $\Q^{p}_{\T'}$, in particular $j$.

Thus we conclude that $\Q^{p}_{\T_{1}} = \Q^{p}_{\T_{2}}$.  

\end{proof}

\begin{rem} \label{sideremark}
Let 
$$\L_{\T} = \begin{bmatrix}
j_{1} & \cdots & j_{r} \\
i_{1} & \cdots & i_{r}
\end{bmatrix}$$

\noindent
be the two-line array obtained by construction the partial $\Q$-symbol as explained above. In particular recall that this means that, at step $j_{s}$, the entry $i_{s}$ was bumped out of the partial $\Q$-symbol at step $j_{s-1}$. First note that 

$$i_{1} \rightarrow \cdots \rightarrow i_{r} \rightarrow \Q^{p}_{\T} $$

\noindent
bumps, by definition, at each step, the entries $i_{s}$ back into the rows out of which they were bumped, and leaves all the other entries in their same row. 
\end{rem}

\begin{lem}\label{bigger} Assume that $j_{s} < j_{t}$ are cancellations in the same column. Recall that in our notation $i_{s}$ is the entry bumped out of the partial $\Q$-symbol at step $j_{s}$. Then 

\begin{align}
\label{columncancellationsareok}
i_{s} > i_{t}.
\end{align}

\end{lem}

\begin{proof} 

This is true due to the ordering in our alphabet $1 < \ldots < n < \bar n < \ldots < \bar 1$ and the semi-standardness of elements in $\op{domres}(\lambda, \mu)$. Indeed, semi-standardness means that if we have an $\bar i$ at step $j_{s}$ and a $\bar j$ at step $j_{t}$

$$\Skew(0:\vdots|0:\mbox{\tiny{$\bar i$}}|0: \vdots |0: \mbox{\tiny{$\bar j$}})$$

\noindent
then the left-most entry in row $i$ is bumped out at step $j_{s}$, and then at step $j_{t}$, the left most entry in row $j$, which is above row $i$ by semi-standardness and the ordering of the alphabet. This implies (\ref{columncancellationsareok}). 
\end{proof}

To see what Lemma \ref{bigger} means for the proof of Proposition \ref{importantprop}, let us look at the following example first. 

\begin{ex}
Let $n \geq 3, \lambda = \omega_{2}+ \omega_{4},$ and $\mu = \omega_{2}$. Then 
$$\T = \Skew(0: \mbox{\tiny{1}}, \mbox{\tiny{1}}|0: \mbox{\tiny{2}}, \mbox{\tiny{2}}|0: \mbox{\tiny{$\bar 2$}}|0: \mbox{\tiny{$\bar 1$}}) \in \op{domres}(\lambda, \mu).$$

\noindent
Below are the sequences of partitions and partial $\Q$-symbols together with the special 2-line arrays at every cancellation.

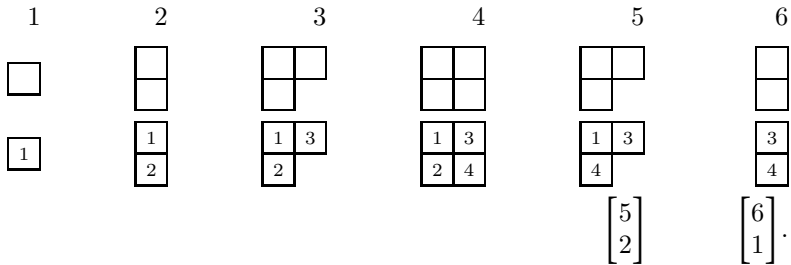
\small{
\begin{align*}
1 &\hbox{ }& 2 & \hbox{ }& 3 &\hbox{ }& 4 &\hbox{ }& 5 &\hbox{ }& 6 \\
 \Skew(0:) &\hbox{ }& \Skew(0: |0: ) & \hbox{ }& \Skew(0:, |0:) &\hbox{ }& \Skew(0: , |0: , ) &\hbox{ }& \Skew(0:, |0: )&\hbox{ }& \Skew(0: |0: ) \\
\Skew(0:\mbox{\tiny{1}}) &\hbox{ }& \Skew(0:\mbox{\tiny{1}}|0: \mbox{\tiny{2}}) &\hbox{ }& \Skew(0:\mbox{\tiny{1}}, \mbox{\tiny{3}}|0:\mbox{\tiny{2}}) &\hbox{ }& \Skew(0:\mbox{\tiny{1}}, \mbox{\tiny{3}}|0:\mbox{\tiny{2}}, \mbox{\tiny{4}}) &\hbox{ }& \Skew(0:\mbox{\tiny{1}}, \mbox{\tiny{3}}|0: \mbox{\tiny{4}})  &\hbox{ }& \Skew(0: \mbox{\tiny{3}}|0: \mbox{\tiny{4}}) \\
 &\hbox{ }&  & \hbox{ }&  &\hbox{ }&  &\hbox{ }& \begin{bmatrix} 5 \\ 2 \end{bmatrix}  &\hbox{ }& \begin{bmatrix} 6 \\ 1 \end{bmatrix}.
\end{align*}
}

\normalsize
\noindent
In this example (using the above notation) $j_{s} = 5, j_{t} = 6, i_{s} = 2,$ and $i_{t} = 1$ (also $i = 2$ and $j = 1$ but this is not important for the example). The fact that $i_{s}> i_{t}$ implies 

\begin{align*}
\Q_{\T} &= 6 \rightarrow 5 \rightarrow 2 \rightarrow 1 \rightarrow \Skew(0: \mbox{\tiny{3}}|0: \mbox{\tiny{4}})  \\
            &= 6 \rightarrow 5 \rightarrow  \Skew(0:\mbox{\tiny{1}}, \mbox{\tiny{3}}|0:\mbox{\tiny{2}}, \mbox{\tiny{4}}) \\
            &= \Skew(0:\mbox{\tiny{1}}, \mbox{\tiny{3}}|0:\mbox{\tiny{2}}, \mbox{\tiny{4}}|0:\mbox{\tiny{5}}|0:\mbox{\tiny{6}}) \quad .
\end{align*}

\noindent

\end{ex}

Note that (\ref{columncancellationsareok}) implies that the steps are bumped back in where they belong, i.e. into the row (with respect to $\T$) in which they are a cancellation. We see in what follows that this is a general phenomenon. For the comfort of the reader, we show the two simplest cases in Lemma \ref{firsteasy} and Lemma \ref{secondeasy} first to illustrate how the proof of Proposition \ref{importantprop} works, and then do the general case.

\begin{lem}\label{firsteasy} If in $\T$ all barred entries occur in the first (left-most) column, then Proposition \ref{importantprop} holds for $\T$.
\end{lem}
\begin{proof}
Let $\T$ be such that all barred entries occur in the first (left-most) column. Let  $d$ be the first cancellation (i.e. the step at which the first deletion occurs.) Since the rest of the barred entries all occur at the end of the column, all entries in the partial $\Q$-symbol at step $d-1$ are strictly smaller that $d$, in particular those in the first column, so bumping these steps back in is, by (\ref{columncancellationsareok}), just putting them back in their column.
\end{proof}

\begin{lem} \label{secondeasy}
If in $\T$ all barred entries occur in the same column, but not the last one, then Proposition \ref{importantprop} holds for $\T$.
\end{lem}
\begin{proof}
Let $\T$ be such that all barred entries are in the same column, but not the last one. Our situation is depicted below, where the picture illustrates the partial $\Q$-symbol at step $d-1$ (the shaded part is not part of the partial $\Q$-symbol but shows where the steps $d$ to $d+k$ appear in $\T$). 

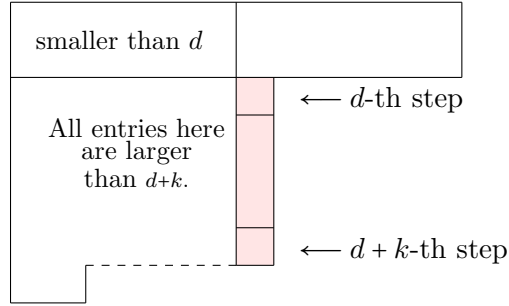
\begin{center}
\begin{tikzpicture}
\fill[red!20!, opacity = 0.5] (3,0.5) rectangle (3.5,3);
\node[right] at (0.2,2) {$\begin{smallmatrix}&\hbox{\small{All entries here}}\\ &\hbox{\small{are larger}}\\ &\hbox{\small{than }}d+k.\end{smallmatrix}$};
\node[right] at (0.2,3.5) {\small{smaller than }$d$};
\node[right] at (3.7,2.7) {$\longleftarrow d$-th step};
\node[right] at (3.7,0.7) {$\longleftarrow d+k$-th step};
\draw[-] (0,0) -- (0,4);
\draw[-] (0,0) -- (1,0);
\draw[-] (1,0) -- (1,0.5);
\draw[dashed] (1,0.5) -- (3,0.5);
\draw[-] (3,0.5) -- (3,4);
\draw[-] (3,0.5) -- (3.5,0.5);
\draw[-] (3.5,0.5) -- (3.5,3);
\draw[-] (3,3) -- (6,3);
\draw[-] (3,2.5) -- (3.5,2.5);
\draw[-] (3,1) -- (3.5,1);
\draw[-] (6,3) -- (6,4);
\draw[-] (0,3) -- (3,3);
\draw[-] (0,4) -- (6,4);
\end{tikzpicture}
\end{center}

Therefore if $\Q^{p, d-1}_{\T}$ is the partial $\Q$-symbol at step $d-1$, and if we let $i_{1}, \ldots, i_{k}$ be the entries bumped out at steps $d, \ldots, d + k$ respectively, then (\ref{columncancellationsareok}) implies:

\begin{align*}
\Q_{\T} = d+k \rightarrow \cdots \rightarrow d \rightarrow i_{1} \rightarrow \cdots \rightarrow i_{k} \rightarrow \Q^{p, d-1}_{\T}.
\end{align*}

Which means that all entries were in fact bumped back into their row.
\end{proof}

\begin{proof}[Proof of Proposition \ref{importantprop}]
We already know that all the steps $i_{1}, \ldots i_{r}$ were bumped back into their row. Let 

\begin{align*}
\Q^{p, half}_{\T} = i_{1} \rightarrow \cdots \rightarrow i_{r} \rightarrow \Q^{p}_{\T}.
\end{align*}

We need to check that when we compute

\begin{align*}
\Q_{\T} = j_{s_{1}} \rightarrow \cdots \rightarrow j_{s_{r}} \rightarrow \Q^{p, half}_{\T},
\end{align*}

\noindent
each step $j_{s_{k}}$ is bumped back into its row (the row in $\T$ in which step $j_{s_{k}}$ of the word-reading occurs). First of all, it is clear that the first of these entries to be bumped in (i.e. entry $j_{s_{r}}$) is indeed bumped into its original row. This is because the entry that it bumped out in the partial $\Q$-symbol process is, by definition, the largest entry to be bumped out. 

Now consider an entry $j_{s_{t}}$ and assume that it is in row $r$. We consider two cases for this entry. The first case is when it represents/ is a cancellation that is the first one to happen in a given column - this means that the first barred entry in that column of $\T$ is in row $r$. It follows that, in $\T$, the entry directly above (hence at an earlier/smaller step) it must be an $r-1$. 

This implies in particular that in the partial $\Q$-symbol up to this point, ($j_{s_{t+1}} \rightarrow \cdots \rightarrow j_{s_{r}} \rightarrow \Q^{p}_{\T}$) there is at least one entry in row $r-1$ that is smaller than $j_{s_{t}}$. Therefore to assure that $j_{s_{t}}$ is bumped into row $r$ we need to check that there are no entries smaller than $j_{s_{t}}$ at that moment. By induction we may assume that all the entries $j_{s_{t+1}}, \ldots , j_{s_{r}}$ were bumped into their rows in 

\begin{align}
\label{inductionpartial}
j_{s_{t+1}} \rightarrow \cdots \rightarrow j_{s_{r}} \rightarrow \Q^{p}_{\T}.
\end{align}

\noindent
By semi-standardness, all entries to the right of the entry at step $j_{s_{t}}$ in $\T$ have to be barred as well. These would be the only candidates to appear in row $r$ at this point. 

However, the entries they cancel out (in the partial $\Q$-symbol process) are, by construction, also smaller than $j_{s_{t}}$, so these entries must belong to the set $\{j_{s_{p}}: p < t\}$, which means that they haven't been bumped into (\ref{inductionpartial}) yet, in particular they cannot be in row $r$. 

The second case is when the entry in $\T$ occurring at the $j_{s_{t}}$-th step is not the first in a column. Then by induction all the entries $j_{s_{t+1}}  \cdots  j_{s_{r}}$ have been bumped into their rows, and by (\ref{columncancellationsareok}) these entries include the steps in the same column and above step $j_{s_{t}}$, so there are entries in the rows $1, \ldots, r -1$, and by the same argument as for the first case above, there can be no entry, in (\ref{inductionpartial}) in row $r$ smaller than $j_{s_{t}}$. 

With this we conclude that the final $\Q$-symbol $\Q_{\T}$ has shape $\lambda$ and depends only on $\lambda$, while the partial $\Q$-symbol $\Q^{p}_{\T}$ depends on $\lambda$ and $\mu$. 
\end{proof}

\begin{proof}[Proof of Theorem \ref{bijection}]
The map $\phi$ is injective: Let $\T_{1} \neq \T_{2}$ be two distinct elements in $\op{domres}(\lambda, \mu)$. By Proposition \ref{importantprop} we know that $\Q_{\T_{1}} = \Q_{\T_{2}}$, which in view of the bijection from Theorem \ref{symplecticrsk}, means that $\phi(\T_{1}) \neq \phi(\T_{2})$, necessarily.

The map $\phi$ is surjective: Let $L \in \op{LRS}(\lambda/\mu)$, and let $\Q_{\lambda}$ be the tableau introduced in Definition \ref{qlambda}. Let 
\begin{align*}
S^{d}_{\mu} = (\mu_{0} = \emptyset, \mu_{1}, \ldots , \mu_{d} = \mu)
\end{align*}

\noindent
be the up-down sequence that corresponds to the pair $(\Q_{\lambda}, L)$ in Theorem \ref{symplecticrsk}. To $S^{d}_{\mu}$ corresponds a unique word w  in the alphabet

\begin{align*}
\mathcal{C}_{n} =  1< \ldots < n < \bar n < \ldots < \bar 1. 
\end{align*}

Let $\Phi^{-1}(L)$ be the filling of the Young diagram of shape $\lambda$ such that its word coincides with w. We want to show that $\Phi^{-1}$ is an inverse, and to show this we need to show that $\Phi^{-1}(L)$ is a semi-standard Young tableau in the alphabet $\mathcal{C}_{n}$ that belongs to the set $\op{domres}(\lambda,\mu)$. The dominance condition is satisfied by definition of the word w.  A case by case analysis assures that this tableau is semi-standard and is therefore an element $\T$ in $\op{domres(\lambda, \mu)}$ such that $\phi(\T) = L$.
\end{proof}

\section{Inequalities for $\op{LRS}(\lambda/ \mu , \eta)$ in the stable case}\label{pol1}
In this section we describe the set $\bigcup_{\eta \text{ even}}\op{LRS}(\lambda/ \mu , \eta)$ for $\mu \subset \lambda$ and $\lambda$ stable (see Definition \ref{stable}) as the set of lattice points of a convex polytope. Note that in this case a Littlewood-Richardson tableau $\T\in \bigcup_{\eta \text{ even}}\op{LR}(\lambda/ \mu , \eta)$ is always $n$-symplectic Sundaram and hence
$$\bigcup_{\eta \text{ even}}\op{LRS}(\lambda/ \mu , \eta)=\bigcup_{\eta \text{ even}}\op{LR}(\lambda/ \mu , \eta).$$

We define a set of inequalities which are equivalent to the ones used to describe Littlewood-Richardson triangles in \cite{pakvallejo}. These are, in turn, in linear bijection to the Berenstein-Zelevinsky triangles in \cite{bztriangles} and the hives of Knutson-Tao \cite{knutsontao}. 

\subsection{The inequalities}
In $\T\in \bigcup_{\eta \text{ even}}\op{LRS}(\lambda/ \mu , \eta)$ entries are weakly increasing along the rows, so $\T$ is determined by the family of numbers

\begin{align}
\label{skewtab}
v(\T) = \{(i,j) \in \mathbb{Z}_{\geq 0}; i \in \{0, 1, \ldots, n\}, j \in  \{1, \ldots, n\}\}
\end{align} 
 
Here, for $$i , j \in \{1, \ldots, n\},$$ \noindent

\begin{align*}
(i,j) := \hbox{ number of $i's$ in row $j$}.
\end{align*}

\noindent
It is straightforward from the definition of $\bigcup_{\eta \text{ even}}\op{LRS}(\lambda/ \mu , \eta)$ that we have the following restrictions on the possible entries in $\T$. 

\begin{lem}
\label{skewvariablelemma}
Let $\T \in\bigcup_{\eta \text{ even}}\op{LRS}(\lambda/ \mu , \eta)$, for some $\lambda$ and for some $\mu$. Then, in $v(\T)$, $(i,j) = 0$ except for possibly $i = j$ and $i = k$ for $1 \leq k \leq j-1$. The vector $v(\T)= \{(i,j); i \in \{1,\ldots,n\}, j \in  \{1, \ldots, n\}\}$ is determined by the set of $n(n+1)$ variables/non-negative numbers which we can picture in an array as below:\\

\begin{center}
$
\begin{array}{llll}
(1,1)& & & \\
(2,2)&(1, 2)& & \\
(3,3)&(2, 3)&(1, 3)& \\
\hbox{    $\hbox{  }$   }\vdots & &\hbox{    $\hbox{  }$   }\vdots & \\
(n,n)& \cdots & & (1,n).
\end{array}
$
\end{center}

\end{lem}

\begin{prop}
\label{skewinequalities}
Let $\T$ be a skew tableau of skew shape $\lambda/ \mu$ and $v(\T) = \{(i,j) \in \mathbb{Z}_{\geq 0}; i \in \{1,\ldots,n\}, j \in  \{1, \ldots, n\}\}$ as in (\ref{skewtab}). Then $\T \in \op{LR}(\lambda/ \mu, \eta)$ for some even $\eta$ if and only if the following inequalities are satisfied: \\

\begin{align} \label{inequ1}
\sum_{p=0}^{i-1}(p,j)& \ge \sum_{p=0}^{i}(p,j+1) \qquad 1 \le i \le j \le n \\ \label{inequ2}
\sum_{q=i}^{j} (i,q) & \ge \sum_{q=i+1}^{j+1}(i+1,q) \qquad 1 \le i \le j \le n \\ \label{inequ3}
\sum_{i=1}^{n}(i,j)&=\sum_{i=1}^{n}(i,j+1) \qquad \text{ for }1\le j \le n, \text{ j odd}, 
\end{align}
where we define $(0,j):=\mu_j$ for all $1\le j \le n$.
\end{prop}

\begin{proof}
The Inequalities \eqref{inequ1} and \eqref{inequ2} ensure by Lemma 3.1 in \cite{pakvallejo} that $T$ is in $\bigcup_{\eta}\op{LRS}(\lambda/ \mu , \eta)$ while Inequality \eqref{inequ3} ensures that $\eta$ is even.
\end{proof}

Let $\mathcal{LR}(\lambda, \mu) \subset \mathbb{R}^{\frac{n(n+1)}{2}}$ be the convex polytope of vectors as in \eqref{skewtab} satisfying the inequalities of Proposition \ref{skewinequalities}. We have the following theorem by Lemma \ref{skewvariablelemma} and Proposition \ref{skewinequalities}.

\begin{thm} Let $\mu \subset \lambda$ be dominant stable weights. The map
\begin{align*}
\displaystyle\bigcup_{\eta \text{ even}}\op{LR}(\lambda/ \mu, \eta) & \rightarrow \mathcal{LR}(\lambda, \mu) \\
 \T & \mapsto v(\T)
\end{align*}
is a bijection between the Littlewood-Richardson-Sundaram tableaux in $$\bigcup_{\eta \text{ even}}\op{LR}(\lambda/ \mu, \eta)$$ and the lattice points of the polytope $ \mathcal{LR}(\lambda, \mu)$.
\end{thm}

\section{Inequalities for Domres in the stable case}
\label{domresviainequalities}

In this section we present a description of the sets $\op{domres}(\lambda, \mu)$, in the case that $\lambda$ is a stable weight, as the set of lattice points of a polytope.

\subsection{The convex polytope associated to $\op{domres}(\lambda, \mu)$}

Since entries are weakly increasing along the rows, each semi-standard Young tableau $\T$ is determined by the family of numbers, parametrized by pairs in $\mathcal{C}_{n}\times \{1, \ldots, n\}$, 
$$w(\T) = \{(i,j); i \in \mathcal{C}_{n}, j \in  \{1, \ldots, n\}\},$$ 
where, for $i \in \mathcal{C}_{n}$ and $j \in \{1, \ldots, n\}$, 

\begin{align*}
(i,j) := \hbox{ number of $i's$ in row $j$}.
\end{align*}

The following lemma is immediate from the definitions.

\begin{lem}
\label{variablelemma}
Let $\T$ be an element of $\op{domres}(\lambda, \mu)$, for some $\lambda$ and for some $\mu$. Then, in $w(\T)$, $(i,j) = 0$ except for possibly $i = j$ and $i =\bar k$ for $1 \leq k \leq j-1$. That is, the vector $w(\T) = \{(i,j); i \in \mathcal{C}_{n}, j \in  \{1, \ldots, n\}\}$ is determined by the family of non-negative numbers, parametrized by the pairs pictured in the array as below:\\

\begin{center}
$
\begin{array}{llll}
(1,1)& & & \\
(2,2)&(\bar 1, 2)& & \\
(3,3)&(\bar 2, 3)&(\bar 1, 3)& \\
\hbox{    $\hbox{  }$   }\vdots & &\hbox{    $\hbox{  }$   }\vdots & \\
(n,n)& \cdots & & (\bar 1,n)
\end{array}
$
\end{center}

\end{lem}

\begin{defn}\label{cancellationproperty}
A semi-standard Young tableau $\T$ as above has the \textit{cancellation property} if, at each step of the word reading, every $\bar{i}$ in the word $W(\T)$ (for $i \in \{1,\ldots, n\}$), can be paired with one $i$ to its left which is not paired with another $\bar{i}$ at a previous step.
\end{defn}
We say that ``all barred letters cancel out,'' or, for each $i$, the $\bar{i}$ always cancels out with an $i$. For example, if $n=3$, the semi-standard Young tableau 
\begin{align*}
\T = \Skew(0:\mbox{\tiny{1}},\mbox{\tiny{2}},\mbox{\tiny{3}}|0:\mbox{\tiny{$\overline{3}$}})
\end{align*}
\noindent has the cancellation property: its word is $W(\T) = 321\bar 3$. The one barred letter in it is $\bar 3$, and there is a $3$ at the beginning with which it cancels out. The semi-standard Young tableau 

\begin{align*}
\T = \Skew(0:\mbox{\tiny{1}},\mbox{\tiny{2}},\mbox{\tiny{$\overline{3}$}}|0:\mbox{\tiny{3}})
\end{align*}
\noindent does not have the cancellation property.

\begin{prop}
\label{cancellinglemma}
A semi-standard Young tableau $\T$ that satisfies the conclusions from Lemma \ref{variablelemma} has the cancellation property if and only if, in $w(\T)$, for every $i \in \{1, \ldots, n\}$ and every $i<k\leq n$ the following inequality holds for $w(\T)$: 

\begin{align}
\label{cancellationinequality}
(i,i) \geq (\bar i, k) + \underset{1 \leq l<  \bar i}{\sum}(l,k).
\end{align}
\end{prop}

\begin{proof}
Let us assume that a certain semi-standard Young tableau $\T$ does not have the cancellation property. Therefore there exists an $i \in \{1, \ldots, n \}$ such that, in $w(\T)$, $(\bar i, k) \neq 0$ for some $i < k \leq n$ (cf. Lemma \ref{variablelemma}) and such that this $\bar i$ is not canceled out (cf. Definition \ref{cancellationproperty}). We assume that $k$ is minimal with this property. Let us consider for a moment the first box filled in with an $\bar i$ that does not cancel out with an $i$ in the word $W(\T)$ of $\T$. Call this box $b$. Let $l \in \mathbb{Z}_{\geq 1}$ be the number of the column that $b$ is at, counting them from left to right. Then 

\begin{align}
\label{thisisthecolumn}
l = (\bar i, k) + \underset{s<\bar i}{\sum}(s,k).
\end{align}

In order to contradict the inequality (\ref{cancellationinequality}) above, let us think of how semi-standard tableaux look like. Assume first that $(a,s) \neq 0 \neq (b,t)$ with $t<s$ and such that the right-most occurrence of $a$ in row $s$ is to the left of or on the column at which the right-most occurrence of $b$ takes place. Then $a<b$ necessarily holds. See Figure 2. In particular, if $(\bar i, s) \neq 0 \neq$ Secondly, observe that whenever there is box in $\T$ filled in with an $i$, all entries to its left must be filled in with an $i$ as well. This follows directly from semi-standardness and from Lemma \ref{variablelemma}, which in particular implies that the only unbarred letter that may occur as an entry in row $i$ is the letter $``i"$ itself.\\
 
The observations in the above paragraph lead to the conclusion that there can be no box filled in with an $i$ above box $b$ in column $l$. This implies 

\begin{align}
\label{noncancellationinequality}
(i,i) < (\bar i, k) + \underset{l<  \bar i}{\sum}(l,k),
\end{align}

\noindent which contradicts (\ref{cancellationinequality}). Now, if an inequality such as (\ref{noncancellationinequality}) holds for some $1\leq i \leq n$ and some $i<k \leq n$, then the previous arguments imply that the semi-standard tableaux $\T$ determined by such numbers $(i,j)$ cannot have the cancellation property. 

\begin{figure}\label{inequ}
$\Skew(0:\hbox{\tiny{a}}|0: \vdots |0: \vdots |0: , \cdots,\cdots, \hbox{\tiny{b}})$
\caption{}
\end{figure}
\end{proof}

\subsection*{Different word readings}
We consider another ``far eastern" word reading: we read rows, from right to left and top to bottom. This defines another word $W'(\T)$ associated to a given semi-standard tableau $\T$. 

\begin{ex}
Consider
\begin{align*}
\T = \Skew(0:\mbox{\tiny{1}},\mbox{\tiny{2}},\mbox{\tiny{3}}|0:\mbox{\tiny{$\overline{3}$}},\mbox{\tiny{$\overline{2}$}}).
\end{align*}
Then we have $W(\T) = 3 2 \bar 2 1 \bar 3$ and $W'(\T) = 321\bar 2 \bar 3$
\end{ex}

\begin{lem}
\label{nondependenceofcancelling}
A semi-standard tableau $\T$ satisfying the conclusions from Lemma \ref{variablelemma} has the cancellation property with respect to $W(\T)$ if and only if it has the cancellation property with respect to $W'(\T)$. 
\end{lem}

\begin{proof}
The proof is by contradiction and quite easy. If $\T$ does not have the cancellation property with respect to $w'(\T)$ minimally at a box $b = \Skew(0:\mbox{\tiny{$\overline{i}$}})$ for some $1 \leq i \leq n$, then it cannot have the cancellation property with respect to $W(\T)$ because there are more $i$'s to the left of the $\overline{i}$ corresponding to $b$ in $W'(\T)$ than in $W(\T)$. Now, if $\T$ does not have the cancellation property with respect to $W(\T)$ minimally at a box $b = \Skew(0:\mbox{\tiny{$\overline{i}$}})$ for some $1 \leq i \leq n$, the conclusions on the variables from Lemma \ref{variablelemma} together with semi-standardness imply that $b$ must belong to row $i$ and hence there can also be no cancellation property with respect to $W(\T)$. 
\end{proof}

\begin{defn} Let $\T$ be a semistandard Young tableaux with entries in the alphabet $\mathcal{C}_{n}$. and $i \in \{1,\ldots, n-1\}$. The \textit{$i$-content} of $W(\T)$ is
\begin{align*}
\# \{ i's \text{ in }W(\T) - \# \{i+1's \text{ in } W(\T)\} + \# \{\overline{i+1} \text{ in } W(\T)\}- \# \{ \overline{i} \text{ in } W(\T)\}.
\end{align*}
The \textit{$i$-content at a given box $b$ of $\T$} is the $i$-content of the subword of $W(\T)$ read (from the left) until the letter filling $b$.
\end{defn}

Recall the notion of the dominance property of a semi-standard tableau $\T$ from Definition \ref{dompropdef}.

\begin{lem}
\label{nondependenceofdominance}
A semi-standard tableau $\T$ satisfying the conclusions from Lemma \ref{variablelemma} with the cancellation property (with respect to either $W(\T)$ or $W'(\T)$ - and therefore both by Lemma \ref{nondependenceofcancelling}) has the dominance property with respect to $W(\T)$ if and only if it has the dominance property with respect to $W'(\T)$. 
\end{lem}

\begin{proof}
Assume first that $W(\T)$ does not have the dominance property, minimally, at some box $b$. By the assumptions on the shape of the tableau $\T$, this must happen at a box $b = \Skew(0:\mbox{\tiny{$\overline{i}$}})$ for some $1 \leq i \leq n$, and by the cancellation property and the conclusions of Lemma \ref{variablelemma}, the box $b$ must be strictly below row $i$. In fact, our tableau $\T$ has the following form (cut at row $i$):
\begin{center}
$\Skew(0:\mbox{\tiny{$i$}}, \mbox{\tiny{$i$}},\mbox{\tiny{$i$}}|0: \mbox{\tiny{i+1}}, \mbox{\tiny{i+1}},\mbox{\tiny{i+1}}|0:\cdots,\cdots,\vdots|0:\cdots,\cdots,\mbox{\tiny{$\overline{i}$}})$.
\end{center}
\noindent
(In the picture, the squares $\Skew(0:\mbox{\tiny{$i$}})$ represent possibly many squares filled in with $i$'s.) Note that since $b$ is minimal with the property that dominance fails at that spot, the squares above box $b = \Skew(0:\mbox{\tiny{$\overline{i}$}})$ cannot be filled in with $\overline{i+1}$. Therefore the boxes that appear to the left of box $b = \Skew(0:\mbox{\tiny{$\overline{i}$}})$ in the word $w'(\T)$ must have entries $k \in \mathcal{C}_{n}$ that satisfy $k < \overline{i+1} < \overline{i}$. In particular $k \neq \overline{i+1}$, and therefore the $i$-content at box $b$ with respect to $W'(\T)$ is the same as the $i$-content with respect to $W(\T)$, and this implies that $W'(\T)$ does not have the dominance property.\\

Now assume the converse is false: $W'(\T)$ does not have the dominance property, and, as before, assume this happens minimally at a box $b = \Skew(0:\mbox{\tiny{$\overline{i}$}})$ (it follows from Lemma \ref{variablelemma} and our assumptions that it must happen at such a box). The rest follows directly by semi-standardness (strictly increasing along the rows): the region of the tableau south-west to box $b$ does not contribute to the $i$-content in $W(\T)$ at box $b$. This implies that $W(\T)$ does not have the dominance property (minimally at $b$ as well). 
\end{proof}

\begin{prop}
\label{dominanceinequalitiesdomres}
Let $\T$ be a semi-standard Young tableau of shape $\lambda$ and content $\mu$ with entries in $\mathcal{C}_{n}$, satisfying the conclusions from Lemma \ref{variablelemma} and such that $W(\T)$ has the cancellation property. Then $\T$ belongs to the set $\op{domres}(\lambda, \mu)$ (i.e. $W(\T)$ has the dominance property) if and only if for every $1 \leq i \leq n-1$ and every $i \leq l \leq n$ the following inequality holds for $w(\T)$

\begin{align}
\label{domresdominance}
(i,i) - (i+1,i+1) - \overset{l}{\underset{k = i}{\sum}}[(\overline{i},k)-(\overline{i+1},k-1)] \geq 0.
\end{align}
\end{prop}

\begin{proof}
We use Lemma \ref{nondependenceofdominance}. Remember that the $i$-content at a given box is
\begin{align*}
\# i's - \# i+1's + \# \overline{i+1} - \# \overline{i}.
\end{align*}

\noindent Once more we proceed by contradiction. First note that for every $1 \leq i \leq n-1$ and $i < l \leq n$ the left hand side of inequality (\ref{dominanceinequalitiesdomres}) is precisely the $i$-content of $\T$ reading the word by rows (i.e. with respect to $W'(\T)$) starting with row $1$ and up until the last $\overline{i}$- entry (reading from right to left) in row $l$. Therefore, by Lemma \ref{nondependenceofdominance}, if in $w(\T)$

\begin{align*}
(i,i) - (i+1,i+1) - \overset{l}{\underset{k = i}{\sum}}[(\overline{i},k)-(\overline{i+1},k-1)] < 0
\end{align*}

\noindent for some such $1 \leq i < l \leq n$, then $W(\T)$ cannot have the dominance property and hence $\T$ does not belong to the set $\op{domres}(\lambda, \mu)$. Now assume that the tableau $\T$ does not belong to the set $\op{domres}(\lambda, \mu)$. As in the proof of Lemma \ref{nondependenceofdominance}, the box $b$ at which non-dominance happens minimally must be of the form $b = \Skew(0:\mbox{\tiny{$\overline{i}$}})$. This means that

\begin{align*}
c_{i}(b) := i\mbox{-content up until and including box } b = -1.
\end{align*}

\noindent Assume box $b$ is in row $l$, for some $i < l$ (it can only happen in such a row, as discussed in the proof of Lemma \ref{nondependenceofdominance}). Then we have in $w(\T)$: 

\begin{align*}
(i,i) - (i+1,i+1) - \overset{l}{\underset{k = i}{\sum}}[(\overline{i},k)-(\overline{i+1},k-1)] < c_{i}(b).
\end{align*}

\end{proof}

For a stable weight $\lambda$ and $\mu\subset \lambda$ let $\mathcal{DR}(\lambda, \mu)\subset \mathbb{R}^{\frac{n(n+1)}{2}}$ be the convex polytope of vectors as in \eqref{skewtab} satisfying the inequalities of Proposition \ref{dominanceinequalitiesdomres}.

By Lemma \ref{skewvariablelemma} and Proposition \ref{dominanceinequalitiesdomres} we have the following theorem.

\begin{thm} Let $\mu \subset \lambda$ be dominant stable weights. The map
\begin{align*}
\op{domres}(\lambda, \mu) & \rightarrow \mathcal{DR}(\lambda, \mu) \\
 \T & \mapsto w(\T)
\end{align*}
is a bijection between the tableaux in $\op{domres}(\lambda, \mu)$ and the lattice points of the polytope $ \mathcal{DR}(\lambda, \mu)$.
\end{thm}

\section{Open Problems}

1. In view of bijection (\ref{thebijection}) it would be interesting to know if there is a unimodular linear map

\begin{align*}
\varphi : \mathbb{R}^{\frac{n(n+1)}{2}} \longrightarrow \mathbb{R}^{\frac{n(n+1)}{2}}
\end{align*}
such that 
$$\varphi(\mathcal{DR}(\lambda, \mu))=\mathcal{LR}(\lambda, \mu).$$

\noindent 
If so, does there exist one that restricts to the bijection in (\ref{thebijection})? For $n=2$, this question has been answered in \cite{jaz}, where such a linear bijection is constructed.\\

\noindent
2. Does Theorem \ref{mainbranchingthm} hold for another path model or even a ``generic'' family of path models? \\

3. Let $s_{\lambda}$ be the Schur function in the variables $x_{1}, \ldots, x_{2n}$. This symmetric polynomial is the character of the representation $\op{L}(\lambda)$. The polynomial $\op{res}(s_{\lambda})$ in the variables $x_{1}, \ldots, x_{n}$ obtained by replacing, in $s_{\lambda}$, the variable $x_{j}$ by $x_{2n-j+1}^{-1}$ for $j > n$ and leaving all the other variables ($x_j$ for $j \leq n$) unchanged is the character of the representation

\begin{align}
\label{localforreader}
\op{res}^{\mathfrak{sl}(2n,\mathbb{C})}_{\mathfrak{sl}(2n,\mathbb{C})^{\sigma}} (\li(\lambda)) = \underset{\op{domres}(\lambda)}{\bigoplus} \tilde \li({\delta(1)}).
\end{align}

\noindent It may be written as

\begin{align}
\label{reschur}
\op{res}(s_{\lambda}) = \underset{\T \in \op{res}(\op{SSYT}(\lambda))}{\sum}x^{\op{content}(\T)},
\end{align}

\noindent
where 

\begin{align*}
x^{(q_{1}, \ldots, q_{n})} := x_{1}^{q_{1}}\cdots  x_{n}^{q_{n}} \hbox{ for } (q_{1}, \ldots, q_{n}) \in \mathbb{Z}^{n} \hbox{ and }
\end{align*}

\begin{align*}
\op{content}(\T) := (a_{1}, \ldots, a_{n}) \hbox{ for } a_{i} = \# i's \hbox{ in } \T = \# \bar i's \hbox{ in } \T. 
\end{align*}

\noindent
Now, in view of (\ref{localforreader}), to each element $\eta$ in $\op{domres}(\eta (1))$ there exists a subset $\R_{\lambda} \subset \op{res}(\op{SSYT}(\lambda))$ such that $\eta$ is the only element in  $\op{domres}(\lambda) \cap \R_{\eta(1)}$, and such that

\begin{align*}
s_{\R_{\lambda}} = \underset{\T \in \R_{\lambda}}{\sum}x^{\op{content}(\T)}
\end{align*}

\noindent 
is the character of the $\mathfrak{sl}(2n,\mathbb{C})$ - module $\li (\tilde{\eta(1)})$. Determining such a subset would give a decomposition 

\begin{align*}
\op{res}(s_{\lambda}) =\underset{\eta \in \op{domres}(\lambda)}{ \sum} R_{\eta(1)}.
\end{align*}

\noindent
A similar problem was proposed by Sundaram in the last chapter of her thesis \cite{sundaram}; she worked in the context of the symplectic tableaux of King \cite{king}.
Introducing analogues of crystal operators on the set $\op{domres}(\lambda)$ would be an interesting approach.\\

\noindent
4. For which pairs $\mathfrak{v} \subset \mathfrak{g}$ of semi-simple Lie algebras does Theorem \ref{mainbranchingthm} hold? For which Littelmann path models?

\section*{Acknowledgments}
Both authors would like to thank the University of Cologne and the Max-Planck Institute in Bonn, where most of this work was carried out. Further we thank Daisuke Sagaki for some interesting discussions and an informative meeting on Skype. J.T. would like to thank Pierre-Emmanuel Chaput for an interesting conversation. During the developing of the present work many computations were performed using \cite{sage} and \cite{polymake} . B.S. was supported by the Japan Society for the Promotion of Science, J.T. was supported by the Max-Planck Institute for Mathematics in Bonn and both authors were partially supported by the Priority Programme SPP 1388 ``Representation theory''.

\bibliographystyle{abbrv}

\end{document}